\documentclass[11pt]{article}

\usepackage{enumerate}
\usepackage{amsmath}
\usepackage{amssymb,latexsym}
\usepackage{amsthm}
\usepackage{color}
\usepackage{graphicx}
\usepackage{amscd}
\usepackage{hyperref}
\usepackage{multicol}
\usepackage{textgreek}
\usepackage{comment}
\usepackage[normalem]{ulem}

\title{Towards the classification of maximum scattered linear sets of $\mathrm{PG}(1,q^5)$}
\author{Stefano Lia, Giovanni Longobardi, Corrado Zanella}
\date{}

\newcommand{\cB}{{\mathcal B}}

\newcommand{\cC}{{\mathcal C}}

\newcommand{\cV}{{\mathcal V}}

\newcommand{\F}{{\mathbb F}}

\newcommand{\Fq}{{\F_q}}
\newcommand{\Fqc}{\F_{q^5}}

\newcommand{\Fqn}{\F_{q^n}}

\newcommand{\la}{\langle}
\newcommand{\ra}{\rangle}

\newtheorem{theorem}{Theorem}[section]

\newtheorem{lemma}[theorem]{Lemma}
\newtheorem{corollary}[theorem]{Corollary}
\newtheorem{proposition}[theorem]{Proposition}

\DeclareMathOperator{\tr}{Tr}

\DeclareMathOperator{\PG}{{PG}}
\DeclareMathOperator{\AG}{{AG}}

\DeclareMathOperator{\PGL}{{PGL}}
\DeclareMathOperator{\PGaL}{P\Gamma L}
\DeclareMathOperator{\GaL}{\Gamma L}

\DeclareMathOperator{\rk}{rk}
\DeclareMathOperator{\N}{N}

\newcommand{\sistlin}[1]{\left\{\begin{array}{ccl}#1\end{array}\right.}

\theoremstyle{definition}%% Hans
\newtheorem{definition}[theorem]{Definition}

\newtheorem{remark}[theorem]{Remark}

\begin{document}

\maketitle

\begin{abstract}
Every maximum scattered linear set in $\PG(1,q^5)$ is the projection of an
$\Fq$-subgeometry $\Sigma$ of $\PG(4,q^5)$ from a plane $\Gamma$ external
to the secant variety to $\Sigma$ \cite{LuPo2004}.
The pair $(\Gamma,\Sigma)$ will be called a projecting configuration for the linear set.
The projecting configurations for the only known maximum scattered linear sets in $\PG(1,q^5)$, 
namely those of pseudoregulus and LP type, 
have been characterized in the literature \cite{CsZa20162,ZZ}.
Let $(\Gamma,\Sigma)$ be a projecting configuration for a maximum scattered linear set in $\PG(1,q^5)$.
Let $\sigma$ be a generator of $\mathbb G=\PGaL(5,q^5)_\Sigma$, and $A=\Gamma\cap\Gamma^{\sigma^4}$,
$B=\Gamma\cap\Gamma^{\sigma^3}$.
If $A$ and $B$ are not both points, then the projected linear set is of pseudoregulus type \cite{CsZa20162}. 
Suppose that they are points.
The rank of a point $X$ is the vectorial dimension of the span of the orbit of $X$ under the action of $\mathbb G$.
In this paper,
by investigating the geometric properties of projecting configurations, it is
proved that if at least one of the points $A$ and $B$ has rank 5, the associated maximum scattered linear set must be of LP type.
Then, if a maximum scattered linear set of a new type exists, it must be such that $\rk A=\rk B=4$. In this paper we derive two possible polynomial forms that such a linear set must have. An exhaustive analysis by computer shows that for $q\le25$ no new maximum scattered linear set exists.
\end{abstract}



\section{Introduction}
\subsection{Linear sets in finite projective spaces}
Let $\PG(d,q^n)=\{\la v\ra_{\Fqn}\colon v\in\F_{q^n}^{d+1},\,v\neq0\}$
be the $d$-dimensional projective space over $\Fq$.
If $U$ is an $r$-dimensional $\Fq$-subspace of $\F_{q^n}^{d+1}$, then
\[ L_U=\{\la v\ra_{\Fqn}\colon v\in U,\,v\neq0\}\subseteq\PG(d,q^n) \]
is called \emph{$\Fq$-linear set} of rank $r$.
Despite the simple definition, linear sets are very rich structures, connected
to interesting objects, such as blocking sets, two-intersection sets, complete caps, 
translation spreads of the Cayley Generalized Hexagon, translation ovoids of polar
spaces, semifield flocks, finite semifields and rank metric codes, see for example \cite{LavrauwVanderVoorde, Po2010, PoZu2020} and their references.\\
The linear set $L_U$ is 
\emph{scattered} if $\dim_{\Fq}\left(\la v\ra_{\Fqn}\cap U\right)\le1$ for any 
$v\in\F_{q^n}^{d+1}$ or, equivalently, if it has size $(q^r-1)/(q-1)$.
A nontrivial scattered $\Fq$-linear set $L_U$ of $\PG(d,q^n)$ 
of highest possible rank is a \emph{maximum scattered $\Fq$-linear set}
(MSLS for short); in this case $U$ is called a \emph{maximum scattered subspace} of 
$\F_{q^n}^{d+1}$. 
The rank $r$ of an MSLS depends on $d$ and $n$. 
It holds
$ r\le\frac{(d+1)n}2 $, see \cite{BlLa2000}, with equality when $(d+1)n$ is even
\cite{BaGiMaPo2018,BlLa2000,CsMaPoZu2017}. Two $\F_q$-linear sets are called \textit{equivalent} (resp.\ \textit{projectively equivalent}) if  there exists a collineation (resp.\ a \textit{projectivity}) of $\PG(d,q^n)$ that maps one onto the other.

The maximum scattered $\Fq$-linear sets in $\PG(1,q^n)$ can be effectively described by 
means  of linearized polynomials.
In fact, every linear set in $\PG(1,q^n)$ of rank $n$ is projectively equivalent 
to
\[  L_f=\{{\la(x,f(x))\ra_{\Fqn}}\colon x\in\F_{q^n}^*\}, \]
where $f(x)=\sum_{i=0}^{n-1}a_ix^{q^i}$ is some $\Fq$-linearized polynomial.
This $L_f$ is a scattered linear set if and only if each ratio $f(x)/x$ occurs for
at most $q-1$ distinct non-zero values of $x \in \F_{q^n}^*$; in other words, if and only if for all 
$y,z\in\F_{q^n}^*$
\begin{equation}\label{e:sp}
\frac{f(y)}y=\frac{f(z)}z\ \Longrightarrow\ \frac yz\in \Fq.
\end{equation}
Any linearized polynomial fulfilling \eqref{e:sp} is called 
\emph{scattered polynomial}.
There are only four known families of scattered polynomials which are defined for
infinitely many values of $n$ \cite{inf5,inf1,inf2,inf3,inf4}.
Of these families, only the following two occur in $\PG(1,q^5)$, the focus of this work:
\begin{enumerate}[$(i)$]
\item $x^{q^s}$, $(s,n)=1$ that is called  scattered polynomial of 
\emph{pseudoregulus type} \cite{BlLa2000};
\item $x^{q^s}+\delta x^{q^{n-s}}$, $n>3$, $(s,n)=1$, $\N_{q^n/q}(\delta)\neq0,1$ that is called 
scattered 
polynomial of \emph{LP type} from the names of Lunardon and Polverino
who introduced it, \cite{LuPo2001,Sh2016}. 
\end{enumerate}
 The corresponding linear sets are called linear sets of \emph{pseudoregulus} and \emph{LP type}, respectively. Further details on scattered polynomials can be found in \cite{long}.

\subsection{Projecting configurations}
A \emph{canonical $\Fq$-subgeometry} of $\PG(d,q^n)$ is a set projectively equivalent 
to the linear set of rank $d+1$ associated with the $\F_q$-subspace $\F_{q}^{d+1}$ (the set of points with coordinates rational over $\Fq$).
The relevance of subgeometries lies in the following result. 
\begin{theorem}\cite{LuPo2004}\label{t:lupo}
If $L_U$ is an $\Fq$-linear set of rank $r$ in $\Lambda=\PG(d,q^n)$
such that $\la L_U\ra=\Lambda$, then there exists
a projective space $\PG(r-1,q^n)\supseteq \Lambda$, an $\Fq$-canonical subgeometry $\Sigma$
of $\PG(r-1,q^n)$, and a complement $\Gamma$ of $\Lambda$ in $\PG(r-1,q^n)$, such that $\Gamma\cap\Sigma=\emptyset$, and $L_U$ is the projection of $\Sigma$
from the vertex $\Gamma$ onto $\Lambda$. 
Conversely, any such projection is a linear set.
\end{theorem}
In this paper any such pair $(\Gamma,\Sigma)$ is called \emph{projecting configuration} for the linear set $L_U$.  
The properties of $L_U$ can be described in terms of such projecting configuration,
the subspace $\Lambda$ being immaterial.
The corresponding projection map is denoted by $\wp_\Gamma$.
In \cite{CsZa20161}, it was proved that 
for some linear sets the projecting configuration is not unique up to collineations.
This issue is further formalized in \cite{CMP}, where the definition of
\emph{$\GaL$-class}
of a linear set $L_U$ is seen from the perspective of projecting configurations.
For a linear set $L_U$ of $\GaL$-class one the projecting configuration is 
unique up to the action of $\GaL(r,q^n)$; in this case $L_U$ is called \emph{simple}.
It has been recently proved in \cite{Giovultimo} that, for any $n \geq 2$, the $\GaL$-class of an MSLS in $\PG(1,q^n)$ not of pseudoregulus type is at most two.\\
According to Theorem~\ref{t:lupo},
the projecting configuration for a {maximum} linear set $L_U$ in $\PG(1,q^n)$, 
i.e.\ a  linear set of rank $n$, is contained in $\PG(n-1,q^n)$, and it has as a vertex 
an $(n-3)$-dimensional projective subspace $\Gamma$ projecting a subgeometry $\Sigma$, 
$\Sigma\cap\Gamma=\emptyset$, on a line $\ell$ such that $\Gamma\cap\ell=\emptyset$.\\
In this paper, the focus is on the mutual position of $\Gamma$ and $\Sigma$; in 
particular, on the behavior of
the vertex $\Gamma$ under the action of the pointwise stabilizer 
$\mathbb G=\PGaL(n,q^n)_\Sigma$. This is a cyclic group of order $n$.

Assume that $(\Gamma,\Sigma)$ is a projecting configuration in $\PG(n-1,q^n)$ for a linear set $L_U$.
Let $\sigma$ be a generator of $\mathbb G=\PGaL(n,q^n)_\Sigma$.
The \emph{rank} of a point $P\in\PG(n-1,q^n)$ with respect to 
the $\Fq$-subgeometry $\Sigma$ is
	\[ \rk P=1+\dim\la P,P^{\sigma},P^{\sigma^2},\ldots,P^{\sigma^{n-1}}\ra, \]
where, from this point on, the angle brackets $\la\ \ra$ without a field as an index will be used to
denote the projective closure of a point set in a projective space.
Equivalently, $\rk P$ is the minimum size of a subset $\mathcal I\subseteq\Sigma$ 
such that $P\in\la\mathcal I\ra$.
The following result holds.
\begin{theorem} \label{th:scatteredness}
Let  $\Sigma$ be an $\F_q$-canonical subgeometry of $\PG(n-1,q^n)$ and let $\Gamma$ and $\Lambda$ be an $(n-3)$-dimensional projective space and a line, respectively such that
$\Gamma \cap \Sigma = \emptyset = \Gamma \cap \Lambda$. Let consider the projection map of $\Sigma$ from $\Gamma$ into $\Lambda$:
\begin{equation}\label{proj}
\wp_\Gamma: P \in \Sigma \longrightarrow \langle P, \Gamma \rangle  \cap \Lambda  \in \Lambda.
\end{equation}
Then following assertions are equivalent:
\begin{itemize}
\item [$(i)$] $\wp_\Gamma(\Sigma)$ is scattered;
\item [$(ii)$] the restriction of $\wp_\Gamma$ to $\Sigma$ is injective;
\item [$(iii)$] every point of $\Gamma$ has rank greater than two.
\end{itemize}
\end{theorem}

In \cite{CsZa20162,CsZa2018,Giovultimo, LavrauwVanderVoorde, MZ2019, ZZ}, several results on the characterization of known maximum scattered linear sets of the projective line have been achieved. These results yield a classification of MSLSs of 
$\PG(1,q^n)$ for $n \leq 4$. In what follows, we collect some of these and outline the current state of the art.\\
Firstly, since all maximum scattered $\mathbb{F}_2$-linear sets of $\PG(1,2^n)$ have size $2^n-1$ and  are equivalent, they are of pseudoregulus type, see also \cite{payne_complete_1971}.\\
It is also straightforward to see that any maximum scattered $\mathbb{F}_q$-linear set of $\PG(1,q^2)$ is a Baer subline, i.e., a (canonical) subgeometry of $\PG(1,q^2)$. \\

In \cite{CsZa20162}, a characterization result for linear sets of pseudoregulus type, obtained via their projecting configurations, is presented.

\begin{theorem}\cite{CsZa20162}\label{t:CsZa20162}
Assume that $\wp_\Gamma(\Sigma)$ is scattered.
Then the following assertions are equivalent:
\begin{itemize}
    \item[$(i)$] $\wp_\Gamma(\Sigma)$ is of pseudoregulus type;
    \item[$(ii)$] $\dim(\Gamma \cap \Gamma^\sigma) = n-4$ for some generator $\sigma$ of $\mathbb{G}$;
    \item[$(iii)$] there exists a point $P$ and a generator $\sigma$ of $\mathbb{G}$ such that $\operatorname{rk} P = n$ and
    \[
        \Gamma = \langle P, P^\sigma, \ldots, P^{\sigma^{n-3}} \rangle.
    \]
\end{itemize}
\end{theorem}

Since a maximum $\mathbb{F}_q$-linear set of $\PG(1,q^3)$ has a rank-three point of $\PG(2,q^3)$ as its projecting vertex, Condition~$(ii)$ of the theorem above is trivially fulfilled. Hence, we obtain the following.

\begin{theorem}
Any maximum scattered $\mathbb{F}_q$-linear set of $\PG(1,q^3)$ is of pseudoregulus type.
\end{theorem}

This result had already been obtained using properties of the stabilizer of a subgeometry $\Sigma$ of $\PG(2, q^3)$; see \cite[Example~5.1, Section~5.2 and Remark~5.6]{CMP} and \cite{LaVan}. \\
Since the projecting configuration of an $\mathbb{F}_q$-linear set of rank $4$ in $\PG(1,q^4)$ is unique up to collineations (\cite[Theorem~4.5]{CMP}), Csajb\'{o}k and Zanella proved in \cite{CsZa2018} the following classification result.

\begin{theorem}\cite{CsZa2018}
The only maximum scattered linear sets in $\PG(1,q^4)$ are those of pseudoregulus type or of LP type.
\end{theorem}

In the same spirit of Theorem \ref{t:CsZa20162}, in \cite{ZZ} the authors provided a characterization of LP type linear sets. Their result, which will be crucial for this article, is the following.

\begin{theorem}\cite{Giovultimo,ZZ}\label{t:ZZ}
Let $n \geq 4$ and $q>2$.
Assume that $\wp_\Gamma(\Sigma)$ is scattered and not of pseudoregulus type.
Then $\wp_\Gamma(\Sigma)$ is of LP type if and only if there exists a generator $\sigma$ of $\mathbb{G}$ such that:
\begin{enumerate}[$(i)$]
    \item $\dim(\Gamma \cap \Gamma^\sigma \cap \Gamma^{\sigma^2}) > n-7$;
    \item there exist points $P$ and $Q$ such that
    \[
        \Gamma = \langle P, P^\sigma, \ldots, P^{\sigma^{n-4}}, Q \rangle,
    \]
    and the line $\langle P^{\sigma^{n-3}}, P^{\sigma^{n-1}} \rangle$ meets $\Gamma$.
\end{enumerate}
Moreover, if $\wp_\Gamma(\Sigma)$ is of LP type, then the point $P$ is unique and has rank $n$.
\end{theorem}

\subsection{Outline of the paper}
In this paper,  we study the case $n=5$.
The projecting configuration $(\Gamma,\Sigma)$ consists
of a {plane $\Gamma$} and an $\Fq$-subgeometry {$\Sigma\cong\PG(4,q)$} 
in $\PG(4,q^5)$.
In \cite{DBVV2022}, the projecting configurations for {non}-scattered linear sets
are investigated.
Section~\ref{s:geomprop} deals with the geometric properties of a projecting configuration
$(\Gamma,\Sigma)$ for an MSLS in $\PG(1,q^5)$.\\
For any point $X\in\PG(4,q^5)$, define $X_i=X^{\sigma^i}$, $i \in \mathbb{N}$. 
If $\wp_\Gamma(\Sigma)\subseteq\PG(1,q^5)$ is an MSLS not of pseudoregulus type,
and $\sigma$ is a generator of $\mathbb G$, then 
\begin{itemize}
\item[$(i)$] the intersection points of $\Gamma$ with $\Gamma^{\sigma^4}$,
$\Gamma^{\sigma}$, $\Gamma^{\sigma^3}$, $\Gamma^{\sigma^2}$ are of type
$A$, $A_1$, $B$, $B_2$ (cf.\ Lemma~\ref{l:AB}), respectively; 
\item[$(ii)$] no three of them are collinear (Proposition~\ref{frame});
\item[$(iii)$] $\rk A\ge4$, $\rk B\ge4$ (Proposition~\ref{linearlyindepedent});
\item[$(iv)$] $\wp_\Gamma(\Sigma)$ is of LP type iff 
$\la A_2,A_4\ra\cap\Gamma\neq\emptyset$ and $\rk A=5$, 
or $\la B_3,B_4\ra\cap\Gamma\neq\emptyset$ and $\rk B=5$ (Theorems~\ref{t:ZZ} and \ref{LPlines}).
If $\rk A=\rk B=4$, then we have a scattered LS of a new type, i.e. not projectively equivalent to the known ones.
\end{itemize}
In Section~\ref{s:equations} algebraic relations between representative vectors for
the points $A_i$ and $B_j$ are proved.
Taking in $\PG(4,q^5)$ coordinates such that
some $A_i$s and $B_j$s are base points, we derive algebraic conditions characterizing the MSLSs of LP 
type.

The investigation is strongly influenced by the rank of the points $A$ and $B$.
Thus, in Section~\ref{s:class} the case of $\max\{\rk A,\rk B\}=5$ is studied.

This work builds on the results in \cite{MZ2019} concerning maximum scattered linear sets of $\PG(1,q^5)$ that are not of pseudoregulus type. In that paper, a necessary and sufficient condition was proven for these sets to be of type LP under the assumption that the point $A$, as described in Section~\ref{s:geomprop}, has rank 5. 
As in \cite{MZ2019}, an algebraic curve $\cV$ of degree at most four is associated to 
$\wp_\Gamma(\Sigma)$.
Up to some special cases, this $\mathcal V$ turns out to be an absolutely irreducible
curve of genus at most three.
In the general case, the  Hasse-Weil Theorem guarantees the existence of a point with coordinates
rational over $\Fq$, while the remaining cases have been analyzed with 
a case-by-case analysis.
It follows that for $\rk A=5$ or  $\rk B=5$ any such MSLS is of LP type.

The case $\rk A=\rk B=4$, for which $\wp_\Gamma(\Sigma)$ would be of a new type, remains open and is partially treated in Section~\ref{s:rkAB4}.
In Subsection~\ref{ss:geo} it is proved that this is equivalent to the existence of a nucleus of a special scattered linear set in $\PG(2,q^5)$.
In Subsection~\ref{ss:alg} the possible new sets are described algebraically (Theorem~\ref{t:fecc}), Proposition~\ref{p:fecc}).
This allowed to run a \texttt{GAP} script showing that there is no new scattered linear set up to $q=25$.

The results of this work are summarized in Theorem~\ref{t:summary}.

\section{Some geometric properties}\label{s:geomprop}
A point $P\in\PG(d,q^n)$ of homogeneous coordinates $a_0,a_1,\ldots,a_d$ (with respect to some given frame) will be denoted by $P=[a_0,a_1,\ldots,a_d]$.
In this section we suppose that
\begin{itemize}
\item [-] $\Sigma \cong \PG(4,q)$ is a canonical $\Fq$-subgeometry in $\PG(4,q^5)$.
When  necessary coordinates will be chosen such that
\[
\Sigma=\{[t,t^q,t^{q^2},t^{q^3},t^{q^4}]\colon t\in\Fqc^*\}.
\]
\item [-] $\Gamma$ is a plane in $\PG(4,q^5)$ such that the projection $\mathbb{L}=\wp_{\Gamma}(\Sigma)$ is a maximum scattered linear set of $\PG(1,q^5)$.
\item [-] $1\neq\sigma \in \PGaL(5,q^5)$ belongs to the subgroup of order five $\mathbb G$  
that fixes $\Sigma$ pointwise.
\item [-] The linear set $\mathbb{L}=\wp_{\Gamma}(\Sigma)$ is not of pseudoregulus type.
\end{itemize}

\begin{lemma}\label{l:AB} 
There exist two points $A,B \in \Gamma$ such that $A^\sigma, B^{\sigma^2} \in \Gamma$ and these are unique.
\end{lemma}
\begin{proof}
Consider the subspaces $\Gamma \cap \Gamma^{\sigma^{4}} $ and $\Gamma \cap \Gamma^{\sigma^3}$. 
By Theorem \ref{t:CsZa20162}, these are necessarily points, say $A$ and $B$, otherwise $\mathbb{L}$ is a linear set of pseudoregulus type.
Now, suppose that $X$ is a point in $\Gamma$ such that $X^\sigma \in \Gamma$, then $X \in \Gamma \cap \Gamma^{\sigma^4}$, so $X=A$. Replacing $\sigma$ by $\sigma^2$, the same argument  shows the uniqueness of point $B$.
\end{proof}

Hereafter, for the remainder of the article, the points $A, B \in \Gamma$ are defined by
\[
A = \Gamma \cap \Gamma^{\sigma^4} \quad \text{and} \quad
B = \Gamma \cap \Gamma^{\sigma^3}.
\] Moreover, the notation $X_i=X^{\sigma^i}$, $i=0,1,\ldots,4$, will be adopted for any 
point $X$ in $\PG(4,q^5)$, and similarly $\Gamma_i=\Gamma^{\sigma^i}$, $i=0,1,\ldots,4$.
Note that, for any $i,j \in \{0,1,2,3,4\}$ with $i\neq j$, $\Gamma_i\cap\Gamma_j=A_h$ or $\Gamma_i\cap\Gamma_j=B_h$ for some $h\in\{0,1,2,3,4\}$.

\begin{proposition} The ten points $A_i,B_i$ for $i=0,1,2,3,4$ are distinct.
\end{proposition}
\begin{proof}
First of all, we show that $A_i \neq A_j$ and $B_i \neq B_j$ for $i,j \in \{0,1,2,3,4\}$ distinct.
Indeed, if $A_i=A_j$ with $0\le i < j\le4$, then $A_i^{\sigma^{j-i}}=A_i$. 
So $A_i$, and hence $A$ belong to $\Sigma$, since they are fixed by $\mathbb{G}$, which leads to a contradiction.
The same argument holds for the points of type $B_i$. 
Next, it is enough to prove that $A$ is distinct from $B_i$ for $i=0,1,2,3,4$.
\begin{enumerate}
\item if $A=B$, then $\Gamma \cap \Gamma_4 = \Gamma \cap \Gamma_3$;
so, $A=\Gamma_3\cap\Gamma_4=(\Gamma\cap\Gamma_4)^{\sigma^4}=A_4$, 
a contradiction.
\item if $A=B_1$ then $\Gamma \cap \Gamma_4 = \Gamma_1 \cap \Gamma_4$;
so, $A=\Gamma\cap\Gamma_1=(\Gamma\cap\Gamma_4)^{\sigma}=A_1$,
a contradiction.
\item if $A=B_2$ then $\Gamma \cap \Gamma_4= \Gamma \cap \Gamma_2$;
so, $B_2=\Gamma_2\cap\Gamma_4=(\Gamma\cap\Gamma_3)^{\sigma^4}=B_4$,
a contradiction.
\item if $A=B_3$ then $B$, $B_2$ and $B_3$ belong to $\Gamma$, and $B,B_2\in\Gamma_2$. Then, $ \langle B, B_2\rangle $, which is a line as proved before, is contained in $\Gamma \cap \Gamma_2$. By Theorem \ref{t:CsZa20162}, this is a contradiction.
\item if $A= B_4$ then $\Gamma \cap \Gamma_4= \Gamma_2 \cap \Gamma_4$:
so, $B_4=\Gamma\cap\Gamma_2=(\Gamma\cap\Gamma_3)^{\sigma^2}=B_2$,
a contradiction.
\end{enumerate}
Then it follows that the ten points $A_i$ and $B_i$, $i=0,1,2,3,4$, are distinct.
\end{proof}

\begin{theorem}\label{LPlines}
 The maximum scattered linear set $\mathbb{L}=\wp_\Gamma(\Sigma)$ is of LP type if and only if $\la A_2,A_4 \ra \cap \Gamma \neq  \emptyset$, or  $\la B_3 ,B_4 \ra  \cap \Gamma \neq \emptyset  $.
\end{theorem}

\begin{proof}
 The result is a direct consequence of Theorem~\ref{t:ZZ}. 
 Indeed, the linear set $\mathbb{L}$ is of LP type if and only if there exists a generator $\tau$ of the group $\mathbb G$  
and a unique point, say $Z$, such that $Z,Z^\tau \in \Gamma$ and 
\begin{equation}\label{equation-1}
\la Z^{\tau^2},Z^{\tau^4} \ra \cap \Gamma \neq \emptyset.
\end{equation}
Now, if $\tau \in \{\sigma, \sigma^4\}$ then $Z\in \{A,A_1\}$ and   $\la Z^{\tau^2},Z^{\tau^4} \ra=\la A_2,A_4 \ra$; if $\tau \in \{\sigma^2, \sigma^3\}$ then $Z \in \{B,B_2\}$ and  $\la Z^{\tau^2},Z^{\tau^4} \ra=\la B_3,B_4 \ra$.
\end{proof}

\begin{definition}
If $\la A_2,A_4 \ra \cap \Gamma \neq  \emptyset$, then
$(\Gamma,\Sigma)$ has \emph{configuration~I}.
If $\la B_3 ,B_4 \ra  \cap \Gamma \neq \emptyset  $,  then 
$(\Gamma,\Sigma)$ has \emph{configuration~II}.
\end{definition}

\begin{proposition}\label{notaligned}
The points $A,A_1,B,B_2$ are not collinear.
\end{proposition}
\begin{proof}
Suppose that there exists a line $\ell$ containing the points $A,A_1,B,B_2$.
First of all, $\ell^\sigma\neq\ell$ since otherwise $\ell$ would meet $\Sigma$ in $q+1$ points.
Since $A_1\in\ell\cap\ell^\sigma$, the span $\pi=\la\ell\cup\ell^\sigma\ra$ is a plane.
The points $A,A_1,A_2,B,B_1,B_2,B_3$ lie on $\pi$; so, 
$\ell^{\sigma^2}=\la A_2,B_2\ra\subseteq\pi$.
It follows $\pi^\sigma=\la\ell^\sigma\cup\ell^{\sigma^2}\ra=\pi$.
Then $\pi \cap \Sigma$ is plane of $\Sigma$. Hence, $\pi$ contains a line $m$ such that $|m \cap \Sigma|=q+1$. 
Therefore, $m$ meets $\ell \subseteq \Gamma$, which, by Theorem~\ref{th:scatteredness} $(iii)$, yields a contradiction.
\end{proof}

As consequence of proposition above, we have $\Gamma=\langle A,A_1,B,B_2 \rangle$.

\begin{lemma}\label{Baligned}
If $B \in \la A,A_1 \ra$,  then
\begin{itemize}
\item [$i)$] $\rk A =5$,
\item [$ii)$] $\mathbb{L}$ is not of LP type.
\end{itemize} 
\end{lemma}
\begin{proof}
If $\rk A<5$, then there is a solid
(i.e.\ a subspace of projective dimension three) $T$  such that $A_i,B_i\in  T$ for all 
$i=0,1,2,3,4$, and $T$ also contains all planes $\Gamma_i$, a contradiction.

Now, since $\rk A=5$, 
$S=\la A,A_1,A_2,A_3\ra$ is a solid.
Since $B \in \langle A, A_1 \rangle$, $B_2\in\la A_2,A_3\ra$ and by Proposition \ref{notaligned}, the plane $\Gamma$ is contained in $S$. Since
$\la A_2,A_4 \ra \cap S= A_2$, and $A_2 \notin \Gamma$ (cf.\ Lemma \ref{l:AB}),
$\la A_2,A_4\ra \cap \Gamma = \emptyset$.

By Theorem \ref{LPlines}, it remains to show that $\la B_3 ,B_4 \ra \cap \Gamma = \emptyset$.
For any $i=1,2,3,4$, the intersection $\Gamma_i \cap S$ is a line. By hypothesis $B \in \langle A,A_1 \rangle$, then the points $A_3,A_4,B_3$ are collinear.
Taking into account $A_4 \not \in S$, $A_3 \in S$, we have that $B_3$ is not in $S$.
In the same way, since $A \in S$, $A_4 \not \in S$ and $A_4,A,B_4$ are collinear, 
$B_4$ is not in $S$.
Consider now a plane $\pi$ containing $A, A_3, A_4,B_3,B_4$. 
Such $\pi$ meets $S$ in the line $\la A, A_3 \ra $ and $\pi \cap \Gamma=A$.  
The line $\la B_3, B_4 \ra \not \subseteq \Gamma_4$ is contained in $\pi$, is distinct from 
$\la A  ,B_4 \ra \subseteq \Gamma_4$, hence
$\la B_3 ,B_4 \ra \cap \Gamma = \emptyset$. 
\end{proof}

\begin{proposition}\label{linearlyindepedent}
Any four points among $A,A_1,A_2,A_3,A_4$ are independent.
Any four points among $B,B_1,B_2,B_3,B_4$ are independent.
\end{proposition}
\begin{proof}
Since $A_2 \not \in \la A,A_1 \ra $, the points $A,A_1,A_2$ are independent. 
If $A_3 \in \la A,A_1,A_2\ra$, then the lines $\la A, A_1 \ra \subseteq \Gamma$ and 
$\la A_2,A_3 \ra \subseteq \Gamma_2$ have a point in common  which must be $B_2$, the intersection point of $\Gamma \cap \Gamma_2$. 
Let $S=\la A,A_1,A_2,A_3,A_4\ra$.
The dimension of $S$ is at most three.
From $B_2\in S$ and $S^\sigma=S$,  it follows that $S$ contains all ten points of type $A_i$ and $B_i$.
Therefore, by Proposition  \ref{notaligned}, $S$ contains all planes $\Gamma_i$, $i=0,1,2,3,4$,
a contradiction.
This implies that $A,A_1,A_2,A_3$ are independent and the first assertion follows.

Similarly, $B$, $B_2$ and $B_4$ are independent. Assume $B_1\in\la B,B_2,B_4\ra$.
Then the lines $\la B,B_2\ra\subseteq\Gamma$ and $\la B_4,B_1\ra\subseteq\Gamma_4$
 meet in the point $A$.
This leads to a contradiction similar to the previous one.
\end{proof}

Throughout this paper, we denote the \textit{trace} and the \textit{norm} functions of $ \F_{q^5}$ over $\F_q$ by
$$\tr_{q^5/q}(x) = x + x^q + x^{q^2} + x^{q^3} + x^{q^4} \textnormal{ and } \N_{q^5/q}(x) = x^{1+q+q^2+q^3+q^4},$$
respectively. For more details on their properties, see \cite[Chapter~2]{LidlNeid}.\\

In \cite[Proposition 3.8]{Z}, the third author showed that $f_b(x) = x^q +b x^{q^2} \in \F_{q^5}[x]$ is non-scattered for any $b \in \F_{q^5}^*$.
The following result allows us to complete the analysis of scatteredness of binomials in $\F_{q^5}[X]$.

\begin{proposition} \label{notscattered}
For any $b \in \F^*_{q^5}$, the linearized polynomial $g_b(x)=x^{q^2}+bx^{q^4}$ is not scattered.
\end{proposition}
\begin{proof}
By Propositions 3.8 and 3.10 in \cite{Z}, since the polynomial $x^q-b^{-1}x^{q^2} \in \F_{q^5}[x]$ is not scattered for any $b \in \F_{q^5}^*$, the algebraic curve 
\begin{equation}\label{chi'b}
    \chi'_b:-bX^{q-1}+Y^{q-1}+1=0
\end{equation} has a point $(\alpha_0,\beta_0)$ in $\AG(2,q^5)$ with coordinates in $\F_{q^5}^*$ such that $\tr_{q^5/q}(\beta_0)=0$. Then
 by \eqref{chi'b},
\begin{equation}\label{alpha0}
    \alpha_0^{q-1}=b^{-1}(\beta_0^{q-1}+1).
\end{equation}
Since $\tr_{q^5/q}(\beta_0)=0$, there exists $z_0 \in \F_{q^5}\setminus \F_q$ such that $\beta_0=z_0^{q}-z_0$.
 Moreover, since for any $x \in \F_{q^5}$, $\N_{q^5/q}(x)^q=\N_{q^5/q}(x)$,
 $$\N_{q^5/q} \left  (-b^{-1} \frac{z_0^{q^2}-z_0}{z_0^{q^4}-z_0} \right )=\N_{q^5/q} \left  (b^{-1} \frac{z_0^{q^2}-z_0}{z_0^{q}-z_0} \right )=\N_{q^5/q}(b^{-1}(\beta_0^{q-1}+1))$$
 and by \eqref{alpha0}, we get 
 $$\N_{q^5/q} \left  (-b \frac{z_0^{q^4}-z_0}{z_0^{q^2}-z_0} \right )=1.$$
 Hence, there exists  $x_0 \in \F_{q^5}^*$ such that $x_0^{q-1}=-b \frac{z_0^{q^4}-z_0}{z_0^{q^2}-z_0}$. Now, putting $y_0=z_0^{q^2}-z_0$ we have that
 $$x_0^{q-1}=-b \frac{z_0^{q^4}-z_0^{q^2}+z_0^{q^2}-z_0}{z_0^{q^2}-z_0}=-b(y_0^{q^2-1}+1).$$
Then, $(x_0,y_0)$ is a point of the  algebraic curve  
\begin{equation*}\label{curve}
\chi_b : b^{-1}X^{q-1}+Y^{q^2-1}+1=0
\end{equation*}
in $\AG(2,q^5)$  with coordinates in $\F^*_{q^5}$ and $\tr_{q^5/q}(y_0)=0$. By  \cite[Proposition 3.10]{Z}, the polynomial $g_b(x)$ is not scattered.
\end{proof}

\begin{proposition} \label{frame}
The quadruple $\mathcal{F}=(A,A_1,B,B_2)$ is a frame for $\Gamma$,
i.e., no three points of $\mathcal F$ are collinear.
\end{proposition}
\begin{proof} 
Let $\bar{\sigma} \in \mathbb{G}$ be the map
\begin{equation*}
    \bar{\sigma}: [x_0,x_1,x_2,x_3,x_4]   \longmapsto [x^q_4,x^q_0,x^q_1,x^q_2,x^q_3].
\end{equation*}
Define $X_{(j)}=X^{\bar\sigma^j}$ for $j=0,1,2,3,4$.
Since $\bar{\sigma}$ is a generator of $\mathbb{G}$, we have $\sigma=\bar{\sigma}^i$ for some $i \in \{1,\dots,4\}$. Hence, the set of points $\{A,A_1,B,B_2\}$ is equal to
$\{\bar{A},\bar{A}_{(1)},\bar{B},\bar{B}_{(2)}\}$, with $\bar{A}=\Gamma \cap \Gamma^{\bar{\sigma}^4}$ and $\bar{B}=\Gamma \cap \Gamma^{\bar{\sigma}^3}$. 
By way of contradiction, suppose that $(\bar{A},\bar{A}_{(1)},\bar{B},\bar{B}_{(2)})$ is not a frame for $\Gamma$. Then it is enough to analyze four cases:
\begin{itemize}
\item [-] \textit{Case 1.} $\bar{B} \in \la \bar{A}, \bar{A}_{(1)} \ra$. Note that a similar result to Lemma \ref{Baligned} also holds for the points $\bar{A}$ and $\bar{B}$ defined by means of $\bar{\sigma}$; indeed, it is enough to follow the same steps of the proof, substituting $\sigma$ by $\bar\sigma$. Therefore, $\rk \bar{A}=5$.
Since the setwise stabilizer of $\Sigma$ in $\PGL(5, q^5)$ acts transitively on the
points of $\PG(4, q^5)$ of rank 5, \cite[Proposition 3.1]{BonPol}, it may be assumed that $\bar{A}$ and  $\bar{A}_{(1)}$ are $[0,1,0,0,0]$ and $[0,0,1,0,0]$, respectively. 
Moreover, since $\bar{B} \in \la \bar{A}, \bar{A}_{(1)} \ra$ and it is distinct from $\bar{A}$ and $\bar{A}_{(1)}$, w.l.o.g.\ we may  assume that $\bar{B}=[0,1,-a^{q^3},0,0]$ with $a \in \F^*_{q^5}$. 
Then
\[
\mathbb{L}=\wp_{\Gamma}(\Sigma)=\{ \la (u, u^{q^4}+au^{q^3})\ra_{\Fqc} : u \in \F^*_{q^5} \}.
\]
Next, $\mathbb L=L_{\hat f}$, where $\hat f(u)=u^{q}+a^{q^2}u^{q^2}$ is the adjoint polynomial of $f(u)=u^{q^4}+au^{q^3}$, \cite[Lemma 3.1]{CMP}.
So,
\[
\mathbb L=\{ \la (u, u^{q}+a^{q^2}u^{q^2})\ra_{\Fqc} : u \in \F^*_{q^5} \},
\]
obtaining a contradiction by \cite[Proposition 3.8]{Z}.
\item [-] $\textit{Case 2}$. $\bar{B} \in \la \bar{A} , \bar{B}_{(2)} \ra$. 
Then $\bar{A} \in \la \bar{B}, \bar{B}_{(2)} \ra$, and $\rk \bar{B} =5$, for otherwise the ten points of type $\bar{A}_{(j)}$ and $\bar{B}_{(j)}$
would belong to a common solid in $\PG(4,q^5)$. 
As before we may choose $\bar{B}_{(2)}=[0,0,0,1,0]$ and $\bar{B}=[0,1,0,0,0]$, respectively, and 
$\bar{A}=[0,-a^{q^4},0,1,0]$ with $a \in \F^*_{q^5}$. 
Then, we get
\begin{equation*}
\mathbb{L}=\{ \la (u, u^{q^2}+au^{q^4})\ra_{\Fqc} : u \in \F^*_{q^5} \},
\end{equation*}
and by Proposition \ref{notscattered}, $\mathbb{L}$ is not scattered, a contradiction.
\end{itemize}
A similar argument can be applied to the cases $\bar{B} \in \la \bar{A}_{(1)},\bar{B}_{(2)} \ra$ and $\bar{B}_{(2)} \in \la \bar{A}, \bar{A}_{(1)} \ra $ 
carrying out to a contradiction. 
Then $\mathcal{F}$ is a frame of $\Gamma$.
\end{proof}

\begin{theorem} Let $E= \la A, B \ra \cap \la A_1, B_2 \ra$ and $F = \la A, B_2 \ra \cap \la A_1, B \ra$. The linear set $\mathbb{L}$ is of LP type if and only if $E \in \la A_2,A_4 \ra $, or $F \in \la B_3, B_4 \ra $.
\end{theorem}
\begin{proof}
The sufficiency of the condition follows directly from Theorem~\ref{LPlines}.

Now assume that $\mathbb L$ is of LP type.
Then at least one among $X=\Gamma\cap\la A_2,A_4\ra$ and $Y=\Gamma\cap\la B_3,B_4\ra$ is a point.

Suppose that $X$ is a point, and $X\neq E$; that is, $\Gamma=\la A,B,X\ra$, or 
$\Gamma=\la A_1,B_2,X\ra$.
Define the solid $S_1=\la\Gamma,A_2,A_4\ra$.
If $\Gamma=\la A,B,X\ra$, then, since $X^{\sigma^2}\in\la A_4,A_1\ra\subseteq S_1$,
also $\Gamma_2=\la A_2,B_2,X^{\sigma^2}\ra$ is contained in $S_1$, hence
$\Gamma\cap \Gamma_2$ is a line: a contradiction.
Similarly, if $\Gamma=\la A_1,B_2,X\ra$, then $S_1$ contains $A_4$, $B$, and
$X^{\sigma^3}\in\la A,A_2\ra$, hence $\Gamma_3\subseteq S_1$, a contradiction again.
We reach analogous contradictions assuming $Y\neq F$.
\end{proof}

Any maximum scattered linear set in $\PG(1,q^5)$ is, up to projectivities, of type $L_{f_0}$ or $L_{f_1}$, with $f_0(x)=x^q+a_2x^{q^2}+a_3x^{q^3}+a_4x^{q^4}$ and
$f_1(x)=x^{q^2}+a_3x^{q^3}$, see \cite[Proposition 2.1]{MZ2019}.
Indeed, if $f(x)=\sum_{i=0}^4a_ix^{q^i}$ and 
$\hat f(x)=\sum_{i=0}^4a_i^{q^{5-i}}x^{q^{5-i}}$, then $L_f=L_{\hat f}$
\cite[Lemma 3.1]{CMP}.
The scattered polynomials of type $f_1(x)$ are of pseudoregulus (when $a_3=0$) or
LP (when $a_3\neq0$) type.
On the other hand, if $f(x)\in\Fqc[x]$ is an $\Fq$-linearized polynomial and
$\Sigma=\{[t,t^q,t^{q^2},t^{q^3},t^{q^4}]\colon t\in\Fqc^*\}$,
then we can explicitly show a vertex $\Gamma_f$ such that $\wp_{\Gamma_f}(\Sigma)\cong L_f$.
In particular,
\[
  \Gamma_{f_0}=\la [0,-a_4,0,0,1],[0,-a_3,0,1,0],[0,-a_2,1,0,0]\ra.
\]
For $q=3,4,5$ all planes $\Gamma_{f_0}$ have been analyzed by a \texttt{GAP} script, 
and when they give rise to MSLSs they always have either the pseudoregulus configuration
(that is, they satisfy the hypotheses of Theorem~\ref{t:CsZa20162}), or  configuration~I, or 
configuration~II.
The scripts for the cases $q=3,4$ and $q=5$ are available at \href{https://pastebin.com/GaV9PTYp}{https://pastebin.com/GaV9PTYp} and \href{https://pastebin.com/TMA5y0Ez}{https://pastebin.com/TMA5y0Ez}, respectively. 
This implies that for $q\le5$, any maximum scattered linear set in $\PG(1,q^5)$ is either of pseudoregulus or of LP type.
This result is contained in our main result, Theorem~\ref{t:summary}, which will be proved with different arguments.

\section{Characterization by equations}\label{s:equations}

In this section, we maintain the notation from the previous one. We will denote by $\mathbb{L}$ the projection $\wp_{\Gamma}(\Sigma)$  with the plane $\Gamma$ as vertex. Recall that $\mathbb{L}$ is a maximum scattered linear set not of pseudoregulus type. Moreover, with a slight abuse of notation, by  $\sigma$ we will denote both 
\begin{itemize}
\item [$i)$] a generator of the subgroup $\mathbb{G}$ of $\PGaL(5,q^5)$ fixing
$\Sigma$ pointwise;
\item [$ii)$] its underlying semilinear map with companion automorphism $x\mapsto x^{q^s}$, $1\leq s\leq 4$.  
\end{itemize}

Note that if $P \in \Sigma$, then there exists a vector $v \in \F^5_{q^5}$ such that 
$P= \la v \ra_{\Fqc}$ and $v^\sigma=v$.
\begin{proposition} \label{prop:linearcombination}
There exist $u,v \in \Fqc^5$ and $\lambda, \mu \in \F^*_{q^5}$ such that $A= \la u \ra_{\Fqc}$, $B=\la v \ra _{\Fqc}$ and
\begin{equation}\label{linearcombination}
v=u-\lambda u^{\sigma}+ \mu v^{\sigma^2}.
\end{equation}
  
If $\N_{q^5/q}(\mu)=1$ or $\N_{q^5/q}(\lambda)=1$, it is possible to choose 
$u$ and $v$ such that $\mu=1$ or $\lambda=1$, respectively.
\end{proposition}

\begin{proof}
Let $(u,v_0) \in \F^5_{q^5} \times \F_{q^5}^5 $ such that $\la u \ra _{\Fqc}=A$ and $\la v_0 \ra_{\Fqc}=B$.
By Proposition \ref{frame}, there exist $\ell,m,n \in \Fqc^*$ such that
\begin{equation*}
v_0=\ell u + m u^\sigma + n v_0^{\sigma^2}.
\end{equation*}
Putting $v= \ell^{-1} v_0$, we have $v^{\sigma^2}= \ell^{-q^{2s}}v^{\sigma^2}_0$ and 
\begin{equation*}
v=u+m\ell^{-1} u^\sigma + n\ell^{q^{2s}-1}v^{\sigma^2}
\end{equation*}
So the first part of the statement follows. Moreover, by Hilbert's Theorem 90, if $\mathrm{N}_{q^5/q}(\mu)=1$ then $\mu=\rho^{\sigma^2-1}$ for some $\rho \in \Fqc^*$, then putting $u'=\rho u$ and  $v'= \rho v$, by \eqref{linearcombination}, we obtain
\begin{equation*}
v'=u'-\lambda' u'^\sigma + v'^{\sigma^2},\qquad\lambda'\in\Fqc^*.
\end{equation*}
A similar argument can be applied in case $\N_{q^5/q}(\lambda)=1$.
\end{proof}

Hereafter, $u,v,\lambda,\mu$ are as specified by \eqref{linearcombination}, and $u_i=u^{\sigma^i}$ and $v_i=v^{\sigma^i}$ for $i \in \{0,1,2,3,4\}$.

\begin{proposition} The following identities hold:
\begin{equation}\label{equation1}
(1-\N_{q^5/q}(\mu))v= \sum_{i=0}^4 a_i u_i\\
\end{equation}
\begin{equation}\label{equation2}
(1-\N_{q^5/q}(\lambda))u= \sum_{i=0}^4 b_i v_i
\end{equation}
where
\begin{multicols}{2}
\noindent $a_0=1-\lambda^{q^{4s}}\mu^{q^{2s}+1}$\\
$a_1=\mu^{q^{4s}+q^{2s}+1}-\lambda$\\
$a_2=\mu(1-\lambda^{q^s} \mu^{q^{4s}+q^{2s}})$\\
$a_3=\mu(\mu^{q^{4s}+q^{2s}+q^s}-\lambda^{q^{2s}})$\\
$a_4=\mu^{q^{2s}+1}(1-\lambda^{q^{3s}}\mu^{q^{4s}+q^s})$
\columnbreak

\noindent$b_0=1-\lambda^{q^{2s}+q^s+1}\mu^{q^{3s}}$\\
$b_1=\lambda(1-\lambda^{q^{3s}+q^{2s}+q^s}\mu^{q^{4s}})$\\
$b_2=\lambda^{q^s+1}-\mu$\\
$b_3=\lambda(\lambda^{q^{2s}+q^s}-\mu^{q^s})$\\
$b_4=\lambda^{q^s+1}(\lambda^{q^{3s}+q^{2s}}-\mu^{q^{2s}})$

\end{multicols}
\end{proposition}
\begin{proof}
By composing
\begin{equation}\label{viui}
v_i=u_i - \lambda^{q^{is}} u_{i+1} + \mu^{q^{is}} v_{i+2},\qquad i=0,2,4,1,3,
\end{equation}
(indices taken mod $5$)
a relation containing $v$ and $u_i$, $i=0,1,2,3,4$ arises, equivalent to \eqref{equation1}.
Similarly, to obtain \eqref{equation2} compose $u_i= v_i - \mu^{q^{is}} v_{i+2} + \lambda^{q^{is}} u_{i+1}$
for $i=0,1,2,3,4$.
\end{proof}

Let take into account the basis $\mathcal{B}=\{u,u_1,u_2,u_3,v_4\}$ of $\Fqc^5$. From now on, if a vector $x \in \F_{q^5}^5$ has coordinates $(a_0, a_1, a_2, a_3, a_4)$ with respect to $\mathcal{B}$, we will denote it by  $x \equiv (a_0, a_1, a_2, a_3, a_4)_{\mathcal{B}}$.
 Then,
\begin{equation*}u_4\equiv (a,b,c,d,e)_{\cB}.
\end{equation*}
for some $a,b,c,d,e \in \F_{q^5}$.
\begin{proposition}\label{relation1}
The following equation holds:
\begin{equation}\label{howe}
  (1-\mu^{q^{4s}+q^s}\lambda^{q^{3s}})e=1-\N_{q^5/q}(\mu).
\end{equation}
\end{proposition}
\begin{proof}
By repeated usage of \eqref{viui},  in coordinates with respect to the base $\mathcal{B}$, we have
\begin{align}
&v_2=u_2-\lambda^{q^{2s}}u_3+\mu^{q^{2s}} v_4\equiv (0,0,1,-\lambda^{q^{2s}},\mu^{q^{2s}})_{\mathcal{B}}\label{howe1}\\
&v=u-\lambda u_1+\mu v_2\equiv(1,-\lambda,\mu,-\lambda^{q^{2s}}\mu,\mu^{q^{2s}+1})_{\mathcal{B}},\label{howe2}
\end{align}
and
\begin{equation}\label{howe3}
u_4=v_4+\lambda^{q^{4s}}u-\mu^{q^{4s}}\left(u_1-\lambda^{q^s}u_2+\mu^{q^s}\left(
u_3-\lambda^{q^{3s}}u_4+\mu^{q^{3s}}(u-\lambda u_1+\mu v_2)
\right) \right).
\end{equation}
The assertion follows by taking into account the fifth coordinates in both expressions in \eqref{howe3}.
\end{proof}

\begin{proposition} The linear set $\mathbb{L}$ is of LP type if and only if 
\begin{equation}\label{B}
(\lambda^{q^{2s}}e+\mu^{q^{2s}}d)(1-\lambda^{q^{3s}}d-\lambda^{q^{3s}+q^{2s}}c)=0
\end{equation}
\end{proposition}
\begin{proof}
By Theorem \ref{LPlines}, the linear set $\mathbb{L}$ is of LP type if and only if $u,u_1,v_2,u_2,u_4$ or $u,u_1,v_2,v_3,v_4$ are linearly dependent.
From \eqref{howe2},
\begin{equation}
\begin{aligned}
&v_3=u_3-\lambda^{q^{3s}}u_4+\mu^{q^{3s}} v \equiv\\
&(-\lambda^{q^{3s}}a+\mu^{q^{3s}},-\lambda^{q^{3s}}b-\lambda\mu^{q^{3s}},-\lambda^{q^{3s}}c+\mu^{q^{3s}+1},1-\lambda^{q^{3s}}d-\lambda^{q^{2s}}\mu^{q^{3s}+1},\\
&-\lambda^{q^{3s}}e+\mu^{q^{3s}+q^{2s}+1})_{\mathcal{B}}.
\end{aligned}
\end{equation}
So, $\mathbb{L}$ is of LP type if and only if one of the following determinants is zero:
\begin{equation}
\begin{vmatrix}
-\lambda^{q^{2s}} & \mu^{q^{2s}}\\
d & e 
\end{vmatrix},
\hspace{1cm}
\begin{vmatrix}
1 & -\lambda^{q^{2s}}\\

-\lambda^{q^{3s}}c+\mu^{q^{3s}+1} & 1- \lambda^{q^{3s}}d-\lambda^{q^{2s}}\mu^{q^{3s}+1}
\end{vmatrix}.
\end{equation}
\end{proof}

\begin{remark}
Consider 
\begin{equation}
v_3=u_3-\lambda^{q^{3s}}u_4+\mu^{q^{3s}}v
\end{equation}
and
\begin{equation}
v_3=\mu^{-q^s}(v_1-u_1+\lambda^{q^s}u_2)=\mu^{-q^{4s}-q^s}(v_4-u_4+\lambda^{q^{4s}}u)-\mu^{-q^s}u_1+\lambda^{q^s}\mu^{-q^s}u_2.
\end{equation}
By difference we get that the vector 
\begin{equation*}
(\lambda^{q^{3s}}-\mu^{-q^{4s}-q^s})u_4
\end{equation*}
has coordinates in $\mathcal{B}$
\begin{gather}
(\mu^{q^{3s}}-\lambda^{q^{4s}}\mu^{-q^{4s}-q^s},-\lambda\mu^{q^{3s}}+\mu^{-q^s},\mu^{q^{3s}+1}-\lambda^{q^s}\mu^{-q^s},
1-\lambda^{q^{2s}}\mu^{q^{3s}+1},\nonumber
\\ \mu^{q^{3s}+q^{2s}+1}-\mu^{-q^{4s}-q^s})_{\mathcal{B}}. \label{coord-u4}
\end{gather}
In case $\lambda\mu^{q^{3s}+q^s}-1 \neq 0$, this allows to obtain the coordinates of $u_4$
in terms of $\lambda$ and $\mu$. 
\end{remark}

Next, we find relations between the parameters $a$, $b$, $c$, $d$, $e$, $\lambda$ and $\mu$ based on the fact that
the semilinear map $\sigma^i$, $i=0,1,2,3,4$, acts on the coordinates as $X\mapsto M_i X^{q^{is}}$.
It holds
\begin{gather}
M_3=\begin{pmatrix}0&a&1&0&0\\ 0&b&0&1&0\\ 0&c&0&0&1\\ 1&d&0&0&-\lambda^{q^{2s}}\\
0&e&0&0&\mu^{q^{2s}}\end{pmatrix},\label{M3}\\
M_2=\begin{pmatrix}0&0&a&1&\mu^{-q^{4s}}(\lambda^{q^{4s}}-a)\\ 0&0&b&0&-\mu^{-q^{4s}}b \label{M2}\\
1&0&c&0&-\mu^{-q^4}c\\ 0&1&d&0&-\mu^{-q^{4s}}d\\ 0&0&e&0&\mu^{-q^{4s}}(1-e)\end{pmatrix}.
\end{gather}
By equating the second columns of $M_2M_3^{q^{2s}}$ and the identity matrix one obtains
the following equations, where $r=c^{q^{2s}}-\mu^{-q^{4s}}e^{q^{2s}}$:
\begin{equation}\label{eq23a}
\sistlin{ar+d^{q^{2s}}+\lambda^{q^{4s}}\mu^{-q^{4s}}e^{q^{2s}}&=&0\\
br&=&1\\
cr+a^{q^{2s}}&=&0\\
dr+b^{q^{2s}}&=&0\\
er+\mu^{-q^{4s}}e^{q^{2s}}&=&0.}
\end{equation}
If $\rk A=5$, or, equivalently, $e\neq0$, the system above is equivalent to
\begin{equation}\label{eq23b}
\sistlin{a&=&\mu^{-q^{2s}-1}e^{1-q^s}(e^{q^s}-1)\\
b&=&-\mu^{q^{4s}}e^{1-q^{2s}}\\
c&=&\mu^{-q^{2s}}e^{1-q^{3s}}(e^{q^{3s}}-1)\\
d&=&-\mu^{q^{4s}+q^s}e^{1-q^{4s}}.}
\end{equation}

\section{\texorpdfstring{Classification in the case $\max \{\rk A,\rk B\}=5$}{Classification in the case max {rk A,rk B\}=5}}}\label{s:class}

Let us define the following linear set:
\[
L_{\alpha,\beta,s}=\{\langle (x-\alpha x^{q^{2s}},x^{q^s}-\beta x^{q^{2s}})\rangle_{\F_{q^5}} \colon x \in \F_{q^5}^*\}\ (\alpha,\beta\in\F_{q^5},\  s\in\{1,2,3,4\}).
\]

The result stated in \cite[Proposition 2.5]{MZ2019} holds more in general for any $1\leq s \leq 4$. More precisely,
\begin{proposition}\label{p:gen1}\,
\begin{itemize}
\item[$(i)$] The linear set $L_{\alpha,\beta,s}$ has rank less than five if and only if
\begin{equation}
\alpha^{q^s} = \beta^{q^s+1} \,\,\textnormal{and}\,\,
  \N_{q^5/q}(\alpha)=\N_{q^5/q}(\beta)=1.  
\end{equation}

\item [$(ii)$] If $\alpha^{q^s} \not = \beta^{q^s+1}$, then $L_{\alpha,\beta,s}$ is not of pseudoregulus type.
\item [$(iii)$] If $\alpha^{q^s} = \beta^{q^s+1}$ and $(\N_{q^5/q}(\alpha)$, $\N_{q^5/q}(\beta)) \not= (1, 1)$, then $L_{\alpha,\beta,s}$ is of pseudoregulus type.
\end{itemize}
\end{proposition}

We point out that the results of \cite[Section 4]{MZ2019} hold true more generally by replacing $\alpha^q$  with $\alpha^{q^s}$ and
$\beta^{q}$ with $\beta^{q^s}$ in their arguments. 
We state here two of them in this general setting.

\begin{proposition}
\label{fulemma41}
Let $\alpha,\beta\in\Fqc$ with $(\alpha,\beta)\neq(0,0)$.
Then $L_{\alpha,\beta,s}$ is maximum scattered if and only if there is no $z\in\Fqc$ such that
\[
\begin{cases}
  \N_{q^5/q}(z)=-1\\
  \beta^{q^{3s}+q^s+1}z^{q^s}-\beta^{q^s+1}z^{q^{2s}+q^s+1}(1-\alpha z)^{q^{3s}}+\beta^{q^{3s}}(1-\alpha z)^{q^s+1}=0.
\end{cases}
\]
\end{proposition}

\begin{proposition}\label{+generale}
Let $\alpha, \beta \in \F_{q^5}$ with $\beta \not = 0$, $1 \leq s \leq 4$ and $\alpha^{q^s}/\beta^{q^s+1} \in \F_{q^5} \setminus \F_q$. Then $L_{\alpha,\beta,s}$ is not maximum scattered.
\end{proposition}

Next, we shall describe the linear set $\mathbb{L}$ using the basis of $\F^5_{q^5}$, $\cB=\{u,u_1,u_2,u_3,v_4\}$ introduced in Section \ref{s:equations} and we will prove that, under the assumption that at least one between $A$ and $B$ has rank 5, $\mathbb{L}$ is of LP type. 
Note that for a generator $\sigma$ of $\mathbb{G}$, the assumption $\max \{\rk A, \rk B\}=5$, where $A= \Gamma \cap \Gamma^{\sigma^4}$ and $B = \Gamma \cap \Gamma^{\sigma^3}$, is equivalent to the existence of $\tau  \in \mathbb{G}$, $\tau\neq1$, such that $\rk A'=5$ with $A'= \Gamma \cap \Gamma^{\tau^4}$ and $A'  \in \{A,B\}$. 
For this reason, it is enough to deal with the case $\rk A=5$.
So, suppose that point $A=\Gamma \cap \Gamma^{\sigma^4}$ has rank $5$. Since $\mathbb{L}$ is scattered, for any two distinct points 
$P_i=\la(x^{(i)}_0,x^{(i)}_1,x^{(i)}_2,x^{(i)}_3,x^{(i)}_4)_{\cB}\ra_{\Fqc}$,
$i=1,2$, in $\Sigma$, the points $A$, $A_1$, $B_2$, $P_1$, $P_2$ are independent;
equivalently, by \eqref{howe1},
\[
\begin{vmatrix} 1&-\lambda^{q^{2s}}&\mu^{q^{2s}}\\ x_2^{(1)}&x_3^{(1)}&x_4^{(1)}\\
x_2^{(2)}&x_3^{(2)}&x_4^{(2)}\end{vmatrix}\neq0.
\]
This implies that the following linear set is scattered:
\begin{equation}\label{fgeneraleA}
  L=\{\la(\mu^{q^{2s}} x_2-x_4,\mu^{q^{2s}}x_3+\lambda^{q^{2s}}x_4)\ra_{\Fqc}\colon
  \la(x_0,\ldots,x_4)_\cB\ra_{\Fqc}\in\Sigma\}.
\end{equation}
The linear set $L$ is the projection of $\Sigma$ from $\Gamma$ to the line $x_0=x_1=x_4=0$.
Therefore, $L\cong\mathbb L=\wp_\Gamma(\Sigma)$.

By  Proposition \ref{prop:linearcombination}, it may be assumed that either $\mu=1$, or $\N_{q^5/q}(\mu)\neq1$ and we have to  distinguish  two cases, that are described in the following propositions.
As in the previous section, let $u_4\equiv(a,b,c,d,e)_{\mathcal B}$, and the semilinear map $\sigma^i$, $i=0,1,2,3,4$, acts on the coordinates as $X\mapsto M_i X^{q^{is}}$.
\begin{proposition}
    If $\mu=1$, then $\mathbb L$ is projectively equivalent to $L_{1,1-e,s}$, and $e\in\Fq$.
\end{proposition}
\begin{proof}
By \eqref{howe}, one has $\lambda=1$ and
\[
a=e-e^{1-q^s},\ b=-e^{1-q^{2s}},\ c=e-e^{1-q^{3s}},\ d=-e^{1-q^{4s}}.
\]
The condition \eqref{B} is equivalent to
\begin{equation}\label{LPcaso2}
(e-1)(e-e^{1-q^{3s}}-e^{1-q^{4s}}-1)=0.
\end{equation}
The equation $M_1X^{q^s}=X$ with $x_4=t$ gives
\[
x_2=e^{-q^{3s}}(t^{q^{3s}}-t)+t,\quad x_3=e^{-q^{4s}}(t^{q^{4s}}-t).
\]
From \eqref{fgeneraleA}, by previous multiplication of the components for $e^{q^{3s}}$ and $e^{q^{4s}}$,
respectively, and by setting $t=x^{q^{2s}}$,
\begin{equation*}
\mathbb L\cong L_{1,1-e^{q^s},s}=\{\la(x-x^{q^{2s}},x^{q^s}-(1-e^{q^s})x^{q^{2s}})\ra_{\Fqc}\colon x \in\Fqc^*\}.
\end{equation*}
By Proposition~\ref{+generale}, $1/(1-e^{q^s})^{q^s+1}\in\Fq$; this implies $1-e^{q^s}\in\Fq$, that is, $e\in\Fq$. In conclusion,
\begin{equation}\label{e:formae}
\mathbb L\cong L_{1,1-e,s}=\{\la(x-x^{q^{2s}},x^{q^s}-(1-e)x^{q^{2s}})\ra_{\Fqc}\colon x \in\Fqc^*\},\quad e\in\Fq.
\end{equation}
\end{proof}

\begin{proposition}
    If $\N_{q^5/q}(\mu) \not = 1$, then $\mathbb L$ is projectively equivalent to $L_{\alpha,\beta,s}$,
where 
\begin{equation}\label{alfbet} 
\alpha=\mu^{-q^{2s}},\quad 
\beta=\frac{\lambda^{q^{2s}}-\mu^{q^{4s}+q^{2s}+q^s}}{\mu^{q^{2s}}(\lambda^{q^{2s}}\mu^{q^{3s}+1}-1)}.
\end{equation}
\end{proposition}
\begin{proof}
Since 
\begin{equation}
M_1=M_3M_3^{q^{3s}}=\begin{pmatrix}0&ab^{q^{3s}}+c^{q^{3s}}&0&a&1\\ 1&b^{q^{3s}+1}+d^{q^{3s}}&0&b&-\lambda\\
0&b^{q^{3s}}c+e^{q^{3s}}&0&c&\mu\\ 0&a^{q^{3s}}+b^{q^{3s}}d-\lambda^{q^{2s}}e^{q^{3s}}&1&d&-\lambda^{q^{2s}}\mu\\
0&b^{q^{3s}}e+\mu^{q^{2s}}e^{q^{3s}}&0&e&\mu^{q^{2s}+1}\end{pmatrix}, \label{M1}
\end{equation}
and $M_1u^\sigma_1=u_2$, one has
\begin{equation}\label{coord-e}
\begin{cases}
&ab^{q^{3s}}+c^{q^{3s}}=0\\
&b^{q^{3s}+1}+d^{q^{3s}}=0\\
&b^{q^{3s}}c+e^{q^{3s}}=1\\
&a^{q^{3s}}+b^{q^{3s}}d-\lambda^{q^{2s}}e^{q^{3s}}=0\\
&b^{q^{3s}}e+\mu^{q^{2s}}e^{q^{3s}}=0.
\end{cases}
\end{equation}
By  $e\neq0$, solving \eqref{coord-e},
\begin{gather*}
a=\lambda^{q^{4s}}e-\mu^{q^{4s}+q^{3s}+q^s}e^{-q^s+1},\ b=-\mu^{q^{4s}}e^{-q^{2s}+1},\\
c=\lambda^{q^s}\mu^{q^{4s}}e-\mu^{q^{4s}+q^{3s}+q^s+1}e^{-q^{3s}+1},\ d=-\mu^{q^{4s}+q^s}e^{-q^{4s}+1},
\end{gather*}
and
\begin{equation}
e=\frac{1-\N_{q^5/q}(\mu)}{1-\lambda^{q^{3s}}\mu^{q^{4s}+q^s}},
\end{equation}
cf. \eqref{howe}.
The equation $M_1X^{q^s}=X$ with $x_4=t$ gives
\begin{align*}
&\mu^{q^{2s}}x_2-t=(1-\N_{q^5/q}(\mu))^{-1}(1-\lambda^{q^s}\mu^{q^{4s}+q^{2s}})(\mu^{q^{2s}}t^{q^{3s}}-t),\\
&\mu^{q^{2s}}x_3+\lambda^{q^{2s}}t=(1-\N_{q^5/q}(\mu))^{-1}\left((\mu^{q^{2s}}-\lambda^{q^{2s}}\mu^{q^{3s}+q^{2s}+1})
t^{q^{4s}}+(\lambda^{q^{2s}}-\mu^{q^{4s}+q^{2s}+q^s})t\right).
\end{align*}
By \eqref{howe}, $\lambda\mu^{q^{3s}+q^s}\neq1$ and
this leads to \eqref{alfbet}.
\end{proof}

\begin{proposition}\label{p:ein}
If $\N_{q^5/q}(\mu)\neq1$, then both $\rho=\mu/\lambda^{q^s+1}$ and 
\begin{equation}\label{howe3case3}  
e=\frac{1-\rho^5\N_{q^5/q}(\lambda)^2}{1-\rho^2\N_{q^5/q}(\lambda)}
\end{equation}
are elements of $\Fq$ with $1 \leq s \leq 4$.
\end{proposition}
\begin{proof}
Let $\alpha$ and $\beta$ be as in \eqref{alfbet}.
If $\beta=0$, then $\lambda^{q^s+1}=\mu\N_{q^5/q}(\mu)$; so, $\rho=\N_{q^5/q}(\mu)^{-1}$.
Otherwise $(\beta^{q^s+1}/\alpha^{q^s})^{q^s-1}=1$ by Proposition~\ref{+generale}.
This is equivalent to 
\[  \frac{\mu^{q^{4s}+q^{2s}}(\lambda^{q^{4s}}-\mu^{q^{4s}+q^{3s}+q^s})(\lambda^{q^{2s}}\mu^{q^{3s}+1}-1)}
{\mu^{q^{4s}+q^{3s}}(\lambda^{q^{4s}}\mu^{q^{2s}+1}-1)(\lambda^{q^{2s}}-\mu^{q^{4s}+q^{2s}+q^s})}=1
\]
and to
$(\N_{q^5/q}(\mu)-1)(\lambda^{q^{4s}}\mu^{q^{2s}}-\lambda^{q^{2s}}\mu^{q^{3s}})=0$,
hence $\lambda^{q^{2s}-1}=\mu^{q^s-1}$.
Finally, Formula \eqref{howe3case3} can be obtained by substitution of $\mu=\rho\lambda^{q^s+1}$ in
\eqref{howe}.
\end{proof}

As seen in \eqref{viui}, $v_4=u_4-\lambda^{q^{4s}}u+\mu^{q^{4s}}v_1$ and this leads to
coordinates $[-a+\lambda^{q^{4s}},-b,-c,-d,1-e]$ of $B_1$ with respect to the basis $\mathcal{B}=\{u,u_1,u_2,u_3,v_4\}$.

In view of Proposition~\ref{p:ein}, if $\N_{q^5/q}(\mu)\neq1$, then \eqref{eq23b} become
\begin{equation}\label{e:23c}
\sistlin{a&=&\mu^{-q^{2s}-1}(e-1)\\
b&=&-\mu^{q^{4s}}\\
c&=&\mu^{-q^{2s}}(e-1)\\
d&=&-\mu^{q^{4s}+q^s}.}
\end{equation}

This implies the following result.
\begin{proposition}\label{p:A3B1}
If $\N_{q^5/q}(\mu)\neq1$, then
the line $A_3B_1$ meets $\Gamma$ in one point.
\end{proposition}
\begin{proof}
Taking into account \eqref{howe1},
the condition of dependence for $A$, $A_1$, $A_3$, $B_1$, $B_2$
is 
\[ \begin{vmatrix}-c&1-e\\ 1&\mu^{q^{2s}}\end{vmatrix}=0. \]
Then the thesis follows from \eqref{e:23c}.
\end{proof}

\subsection{An algebraic condition for being non-LP type}
Until the end of this subsection we will assume that $\mathbb L$ is not of LP type as well as not of pseudoregulus type. 
Also, we will assume that $\rk A=5$.
This allows us to take $[x_0,x_1,x_2,x_3,x_4]$ as homogeneous coordinates of a point $\la\sum x_iu_i\ra_{\Fqc}$.

\begin{proposition}\label{p:almost}
  The linear set $\mathbb L$ is projectively equivalent to
\begin{equation}\label{e:formak}
  \{\la(x-k^{-1}x^{q^{2s}},x^{q^s}-\delta k^{q^{4s}+q^{2s}}x^{q^{2s}})\ra_{\Fqc}\colon
  x\in\Fqc^*\}
\end{equation}
for some $k\in\Fqc^*$ and $\delta\in\Fq$;
that is, $\mathbb L\cong L_{\alpha,\beta,s}$ with
$\alpha=k^{-1}$, $\beta=\delta k^{q^{4s}+q^{2s}}$.
\end{proposition}

\begin{proof}
\textsc{Case $\N_{q^5/q}(\mu)\neq1$.} 
Let  $S=\la\Gamma,A_3\ra$.
The point $A_4$ does not belong to $S$, for otherwise $S$ would also contain 
$\Gamma_4=\la A,A_4,B_1\ra$.
By a similar argument $S$ does not contain $A_2$.
Then the equation of $S$ is $x_4=kx_2$, for some  $k\neq0$ in $\F_{q^5}$.

Let $B=[b_0,b_1,b_2,b_3,b_4]$.
By Proposition~\ref{p:A3B1}, $B_1\in S$.
The algebraic conditions for $B,B_1,B_2\in S$, and $B_2\in \la A,A_1,B\ra$
are $b_4=kb_2$, $b_3^{q^s}=kb_1^{q^s}$, $b_2^{q^{2s}}=kb_0^{q^{2s}}$, and
\[
\rk \begin{pmatrix} b_2&k^{q^{4s}}b_1&kb_2\\ b_0^{q^{2s}}&b_1^{q^{2s}}&b_2^{q^{2s}}\end{pmatrix}=1
\]
respectively.
The coordinates $b_1$ and $b_3$ are nonzero, for otherwise $B\in\la A,A_2,A_4\ra$,
and $\mathbb L$ is of LP type; furthermore, $b_0b_2b_4\neq0$, for otherwise
$A_3\in A_1 B\subseteq\Gamma$, a contradiction.
The condition on the rank implies $b_0=k^{-q^{3s}}b_2$ and $k^{q^{4s}-1}b_2^{q^{2s}-1}=b_1^{q^{2s}-1}$,
hence $b_1=\delta k^{q^{2s}+1}b_2$ for some $\delta\in\F_q^*$.
Therefore $B=[k^{-q^{3s}},\delta k^{q^{2s}+1},1,\delta k^{q^{4s}+q^{2s}+1},k]$, and
\begin{equation}\label{e:coof}
F=AB_2\cap A_1B=[k^{-q^{3s}-1},k^{q^{2s}},k^{-1},\delta k^{q^{4s}+q^{2s}},1].
\end{equation}
By projecting a point $[x^{q^{3s}},x^{q^{4s}},x,x^{q^s},x^{q^{2s}}]$ of $\Sigma$ from the vertex
\[
  \la A,A_1,B_2\ra=
\la[1,0,0,0,0],[0,1,0,0,0],
[*,*,k^{-1},\delta k^{q^{4s}+q^{2s}},1]\rangle\]
onto a complementary line, one obtains \eqref{e:formak}.

\noindent
\textsc{Case $\mu=1$.} 
The linear set
$\mathbb L$ is described in \eqref{e:formae}, which is equivalent to \eqref{e:formak} for $k=1$, $\delta=1-e$.
\end{proof}

\begin{theorem}\label{t:almost}
  Let $\mathbb L$ as described in Proposition~\ref{p:almost}. If $\varepsilon=\N_{q^5/q}(k)$, then $\delta^2\varepsilon\neq1$,
  $\delta^3\varepsilon^2+(1-3\delta)\varepsilon+1\neq0$, and no $x\in\Fqc$ exists
  satisfying
  \begin{equation}\label{e:sctness}
\begin{cases}
  \varepsilon\N_{q^5/q}(x)=-1\\
  \delta^2\varepsilon x^{q^s}-\delta\varepsilon x^{q^{2s}+q^s+1}(1-x)^{q^{3s}}+(1-x)^{q^s+1}
  =0.
\end{cases}
\end{equation}
\end{theorem}

\begin{proof}
\textsc{Case $\N_{q^5/q}(\mu)\neq1$.} 
By Proposition~\ref{p:gen1} ($ii$),
\[
  \frac{\alpha^{q^s}}{\beta^{q^s+1}}=\frac{k^{-q^s}}{\delta^2k^{q^{4s}+q^{3s}+q^{2s}+1}}=
  \frac1{\delta^2\varepsilon}
\]
is distinct from one, so $\delta^2\varepsilon\neq1$.

By substitution of $\beta^{q^{3s}+q^s+1}=\delta^3k^{q^{4s}+q^{3s}+2q^{2s}+2}$,$\beta^{q^{3s}}=\delta k^{q^{2s}+1}$, $\alpha^{q^{3s}}=k^{-q^{3s}}$, $\beta^{q^s+1}=\delta^2k^{q^{4s}+q^{3s}+q^{2s}+1}$,
Proposition~\ref{fulemma41} reads as follows: $\mathbb{L}$ is maximum scattered if and only if there is no $z\in\Fqc$ such that
\[
\begin{cases}
  \N_{q^5/q}(z)=-1\\
\delta^3k^{q^{4s}+q^{3s}+2q^{2s}+2}z^{q^s}-\delta^2k^{q^{4s}+q^{3s}+q^{2s}+1}z^{q^{2s}+q^s+1}(1-k^{-1} z)^{q^{3s}}\\
\quad +\delta k^{q^{2s}+1}(1-k^{-1} z)^{q^s+1}=0.
\end{cases}
\]
Dividing by $\delta k^{q^{2s}+1}$ and substituting $z=kx$, we get \eqref{e:sctness}.

Since $\mathbb L$ is not of LP type, the points \begin{align*}
    B_3&=[1,\delta k^{q^{3s}+q^{2s}+1},k^{q^{3s}},k^{-q^s},\delta k^{q^{3s}+1}], \\
    B_4&=[\delta k^{q^{4s}+q^s},1,\delta k^{q^{4s}+q^{3s}+q^s},k^{q^{4s}},k^{-q^{2s}}],
\end{align*}
and $F$~ (cf. \eqref{e:coof}) are not collinear.
By standard row-reducing, this reads as $\delta^3\varepsilon^2+(1-3\delta)\varepsilon+1\neq0$.

\noindent
\textsc{Case $\mu=1$.}
Clearly condition $\delta^2 \varepsilon \neq 1$ and there is no $x \in \F_{q^5}$ satisfying \eqref{e:sctness} hold. The condition $\delta^3\varepsilon^2+(1-3\delta)\varepsilon+1=0$ is equivalent to $e=3$, and this by \eqref{LPcaso2} implies that $\mathbb L$ is of LP type, a contradiction.
\end{proof}

\section{An algebraic equation}\label{s:conj}

In this section, we shall show that  if $\mathbb{L}$ is a maximum scattered linear set of $\PG(1,q^5)$ neither of pseudoregulus type nor of LP type, and $\rk A=5$, the conditions in Theorem~\ref{t:almost} never hold.
As a consequence, we will get that there are no new maximum scattered linear sets in $\PG(1,q^5)$ with $\max\{\rk A,\rk B\}=5$. More precisely, we will show the following.
\begin{theorem}\label{t:conj}
Let $1 \leq s \leq 4$ and $\delta, \varepsilon \in\mathbb{F}_q^*$  such that $\delta^2\varepsilon\neq1$. If $\delta^3\varepsilon^2 + (1 -3\delta)\varepsilon + 1  \neq 0$, then there exists $x \in \F_{q^5}$ satisfying $\eqref{e:sctness}$.
\end{theorem}

\medskip
Although we will use the techniques contained in \cite[Section 4]{MZ2019}, to make the work as self-contained as possible, we prefer to  adapt them to our context.\\
Consider a normal element of $\F_{q^5}$ over $\F_q$, say $\gamma$, see e.g. \cite[Theorem 2.35]{LidlNeid}. Then, for any integer $1 \leq s \leq 4$, $ \{\gamma,\gamma^{q^s},\gamma^{q^{2s}},\gamma^{q^{3s}},\gamma^{q^{4s}}\}$ is an $\F_q$-basis of $\F_{q^5}$ and every element $x$ in $\F_{q^5}$ can be written as  $x= \sum_{i=0}^{4} x_i\gamma^{q^{si}}$,
where $x_i \in \F_q$, $i=0,1,2,3,4$. Moreover, we can identify $\mathbb{F}_{q^5}$ with $\mathbb{F}_q^5$ in the natural way.

Equations \eqref{e:sctness} are therefore equivalent to a system $\mathcal{S}$ of ten equations $C_j(x_0,x_1,x_2,x_3,x_4)=0$ in the new variables $x_i$.
Denoting by $\mathcal{V}\subseteq \mathrm{AG}(5,q)$ the affine variety associated with $\mathcal{S}$, Theorem \ref{t:conj} is therefore equivalent to the statement that $\mathcal{V}$ has an $\mathbb{F}_q$-rational point. 
To prove this, we will show that $\mathcal{V}$ is a variety of dimension one, i.e.\ an algebraic curve, and then the existence of a point will be a consequence of the Hasse-Weil Theorem , in the case the curve $\mathcal{V}$ is absolutely irreducible \cite[Theorem 9.18]{HirschKorTor}, or an $\F_q$-rational point of $\mathcal{V}$ will be directly exhibited.

We first apply the following change of variables in $\AG(5,\overline{\F}_{q})$ (whose matrix is a so-called Moore matrix and is nonsingular)
\begin{equation*}   
\phi:
\begin{cases}
A=x_0\gamma+x_1\gamma^{q^s}+x_2\gamma^{q^{2s}}+x_3\gamma^{q^{3s}}+x_4\gamma^{q^{4s}}\\
B=x_4\gamma+x_0\gamma^{q^s}+x_1\gamma^{q^{2s}}+x_2\gamma^{q^{3s}}+x_3\gamma^{q^{4s}}\\
C=x_3\gamma+x_4\gamma^{q^s}+x_0\gamma^{q^{2s}}+x_1\gamma^{q^{3s}}+x_2\gamma^{q^{4s}}\\
D=x_2\gamma+x_3\gamma^{q^s}+x_4\gamma^{q^{2s}}+x_0\gamma^{q^{3s}}+x_1\gamma^{q^{4s}}\\
E=x_1\gamma+x_2\gamma^{q^s}+x_3\gamma^{q^{2s}}+x_4\gamma^{q^{3s}}+x_0\gamma^{q^{4s}}.
\end{cases}
\end{equation*}
 
The $\Fq$-rational points in $\AG(5,q)$ are mapped by $\phi$ to those of type $(A,B,C,D,E)=(x,x^{q^s},x^{q^{2s}},x^{q^{3s}},x^{q^{4s}})$.
Denote by $\cC$ the image of $\mathcal{V}$ under $\phi$.
Since the dimension, the genus and the absolute irreducibility are birational invariants, we can study $\cC$ in place of $\cV$.\\
Let
\begin{equation}
\begin{split}
&G(A,B,C,D,E):=ABCDE\varepsilon+1,\\
&F_0(A,B,C,D,E):=\delta^2  \varepsilon  B-\delta  \varepsilon  A B C (1-D)+(1-A) (1-B)\\   
& 
F_{i+1}(A,B,C,D,E):=F_i(B,C,D,E,A), \quad i=0,1,2,3.
\end{split}
\end{equation}

By \eqref{e:sctness}, we get  that $\mathcal{C}$ is given by the following system:
\begin{equation}\label{equazioni}
\mathcal{C}:\begin{cases}
&G(A,B,C,D,E)=0\\
&F_i(A,B,C,D,E)=0,\, i=0,1,2,3,4.
\end{cases}
\end{equation}

Now, we are in the position to prove the following.
\begin{lemma}\label{quartic curve}
    Let $\delta,\varepsilon \in \F_{q}^*$ such that $\delta^2 \varepsilon \neq 1$. If $\delta^3\varepsilon^2 + (1 -3\delta)\varepsilon + 1  \neq 0$, then  the algebraic variety $\cC$ (cf. \eqref{equazioni}) is birationally equivalent to the plane curve
    $\mathcal{Q}: f(X,Y) = 0$ of degree at most four, where
    \begin{equation} \small\label{equazione curva}
      f(X,Y)
      = f_0 X^2 Y^2
        + f_1 X Y (X+Y)
        + f_2 (X^2+Y^2)
        + f_3 X Y
        + f_4 (X+Y)
        + f_5,
    \end{equation}
    with
    \begin{align*}
      f_0 &= (-\delta^2 \varepsilon +1) (\delta\varepsilon-1),\\
      f_1 &= (\delta^2\varepsilon+ \delta\varepsilon- 2) (\delta^2\varepsilon- 1),\\
      f_2 &= -(\delta^2\varepsilon-1)^2,\\
      f_3 &= \delta^6\varepsilon^3-5 \delta^4\varepsilon^2+6 \delta^2\varepsilon+\delta\varepsilon+\delta-4,\\
      f_4 &= (\delta^2\varepsilon-1) (\delta^2\varepsilon+\delta-2),\\
      f_5 &= (\delta-1) (-\delta^2\varepsilon+1).
    \end{align*}

\end{lemma}
\begin{proof}
We eliminate the variables $ C,D,E$ successively from the system \eqref{equazioni}, using 
$G(A,B,C,D,E)=0$ and $F_i(A,B,C,D,E)=0$, to obtain a single equation in $A,B$.
By the first equation in \eqref{equazioni}, we get $A,B,C,D,E\neq 0$ and, hence,
\begin{equation}\label{variableE}
E=-1/(\varepsilon ABCD).
\end{equation}
Moreover, by $F_0(A,B,C,D,E)=0$ we derive
\begin{equation} \label{variableD}
D=P(A,B,C):= (A B C \delta  \varepsilon  - A B + A - B \delta^2  \varepsilon  + B - 1)/(A B C \delta  \varepsilon ).
\end{equation}
Since $F_1(A,B,C,P(A,B,C),-1/( \varepsilon  A B C \cdot P(A,B,C) ))=0$,
we get 
\begin{equation*}
A( B \delta  \varepsilon  -  B -  \delta^2  \varepsilon  + 1 )C=( B \delta^2  \varepsilon  - B - \delta + 1).
\end{equation*}
Note that, since $AC\neq 0$, $ B \delta  \varepsilon  -  B -  \delta^2  \varepsilon  + 1$ is zero if and only if $ B \delta^2  \varepsilon  - B - \delta + 1$ is zero.
In this case, since $B\neq 0$ and $\delta^2\varepsilon\neq 1$, it would follow $(\delta^2\varepsilon-1)^2-(\delta-1)(\delta \varepsilon -1)=0$, against our hypotheses. 
Hence, we have
\begin{equation}\label{variableC}
C=Q(A,B):=( B \delta^2  \varepsilon  - B - \delta + 1)/(A( B \delta  \varepsilon  -  B -  \delta^2  \varepsilon  + 1 )).
\end{equation}

In order to eliminate the variable $C$ from \eqref{variableD}, write
\begin{equation*}
\begin{split}
\hat{P}(A,B):=P(A,B,Q(A,B))
=& \frac{1}{B \delta \varepsilon (1-B-\delta + 
    B \delta^2\varepsilon)}(-A B^2 \delta  \varepsilon  + A B^2 + A B \delta^2  \varepsilon+  \\ 
    & A B \delta  \varepsilon  - 2 A B - A \delta^2  \varepsilon + 
    A + B^2 \delta^2  \varepsilon  - B^2 + B \delta^4  \varepsilon ^2 -\\
    &3 B \delta^2  \varepsilon  + 2 B + \delta^2  \varepsilon  -1).
    \end{split}
\end{equation*}
 Now, we will express the equations $F_i(A,B,C,D,E)=0$, $i \in \{2,3,4\}$, in terms of the variables $A$ and $B$. 

Substituting in $F_2(A,B,C,D,E)=0$, we obtain the equation $f(A,B)=0$, where
\begin{equation*}
f(A,B):=F_2(A,B,Q(A,B),\hat{P}(A,B),-1/( \varepsilon A B Q(A,B) \hat{P}(A,B))).
\end{equation*}
This can be written as
\begin{align*}
f(A,B)=&A^2 B^2 (-\delta^2 \varepsilon +1) (\delta\varepsilon-1) +\\
&A B (A+B) (\delta^2\varepsilon+ \delta\varepsilon- 2) (\delta^2\varepsilon- 1)+\\
&(A^2+B^2) (-(\delta^2\varepsilon-1)^2) +\\
&A B (\delta^6\varepsilon^3-5 \delta^4\varepsilon^2+6 \delta^2\varepsilon+\delta\varepsilon+\delta-4)+\\ 
&(A+B) (\delta^2\varepsilon-1) (\delta^2\varepsilon+\delta-2)+\\
&(\delta-1)(-\delta^2\varepsilon+1).
\end{align*}

Finally, note that both $F_3(A,B,C,D,E)=0$ and $F_4(A,B,C,D,E)=0$, when expressed in terms of $A$ and $B$ only, are satisfied as soon as $f(A,B)=0$.
\end{proof}

\begin{lemma}\label{no-linear-components}
Let $\delta,\varepsilon \in \F_{q}^*$ satisfy $\delta^2 \varepsilon \neq 1$ and
$\delta^3\varepsilon^2 + (1 -3\delta)\varepsilon + 1  \neq 0$. Then, the curve
$\mathcal{Q}: f(X,Y)=0$ (cf.\ \eqref{equazione curva}) has
no linear components.
\end{lemma}
\begin{proof}
The coefficient $f_0$ of the term of degree four of $f(X,Y)$ in \eqref{equazione curva} is 0  if and only if $\delta\varepsilon=1$. We split the proof in two cases: $f_0=0$ and $f_0\neq 0$.\\
\textbf{Case 1.} $f_0 = 0$. We have that $\delta-1\neq 0$, otherwise $\delta^2\varepsilon=1$. Then,
$$f(X,Y)=(\delta - 1)^2(X^2Y - X^2 + XY^2 + XY\delta - 3XY + 2X - Y^2 + 2Y - 1),$$
and $\mathcal{Q}$ is a cubic of the affine plane $\AG(2,q)$. Let $[\overline{X},\overline{Y},\overline{Z}]$ the homogeneous coordinates of a point in $\PG(2,q)$ and let $g(\overline{X},\overline{Y},\overline{Z}):=\overline{Z}^3f(\overline{X}/\overline{Z},\overline{Y}/\overline{Z})$. We now show that the projective variety $\overline{\mathcal{Q}}: g(\overline{X},\overline{Y},\overline{Z})=0$ has no linear components. 
Note that the intersection of the curve $\overline{\mathcal{Q}}$ with the coordinate axes $\overline{X}=0$, $\overline{Y}=0$ and $\overline{Z}=0$ are the point sets $I_1=\{[0,1,0],[0,1,1]\}$, $I_2=\{[1,0,0],[1,0,1]\}$ and $I_3=\{[0,1,0],[1,0,0],[1,-1,0]\}$, respectively. 
If a line $\ell$ is a component of $\overline{\mathcal{Q}}$, then $\ell$ must meet each of the coordinate axes in at least one point of $I_i$, $i=1,2,3$. 
This follows from the fact that any two lines in the plane meet in a point or they coincide, and that the intersection of $\ell$ with the coordinates axes must be a subset of the the intersection of $\overline{\mathcal{Q}}$ with the axes.
In particular, $\ell$ must be the line joining  a point one of $I_1$ and a point of $I_2$.
By direct checking, a line joining a point of $I_1$ and a point of $ I_2$ is one of the following:
\begin{itemize}
\item [$(i)$] the line $\overline{Z}=0$. This is not a component of $\overline{\mathcal{Q}}$.
\item [$(ii)$] the line $\overline{X}-\overline{Z}=0$. If this line is a component of $\overline{\mathcal{Q}}$, then 
\begin{align*}
& \overline{X}^2\overline{Y}+\overline{X}\overline{Y}^2-\overline{X}^2\overline{Z}+(\delta-3)\overline{X}\overline{Y}\overline{Z}+2\overline{X}\overline{Z}^2-\overline{Y}^2\overline{Z}+2\overline{Y}\overline{Z}^2-\overline{Z}^3=\\
& (\overline{Z}-\overline{X})(a\overline{X}^2+b\overline{Y}^2+c\overline{X}\overline{Y}+d\overline{X}\overline{Z}+e\overline{Y}\overline{Z}+f\overline{Z}^2)
\end{align*}
for some $a,b,c,d,e,f \in \F_{q}$. Then,  comparing the coefficients on the left and the right hand-side, we get $a=0$, $b=c=-d=f=-1$ and $e=2$. This implies $\delta=0$, which is a contradiction with our hypotheses. 
\item [$(iii)$] the line $\overline{Y}-\overline{Z}=0$.  Since $g(\overline{X},\overline{Y},\overline{Z})=g(\overline{Y},\overline{X},\overline{Z})$, a contradiction follows as above.
\item [$(iv)$] the line $\overline{X}+\overline{Y}-\overline{Z}=0$.  If this line is a component of $\overline{\mathcal{Q}}$, then 
\begin{align*}
&(\overline{X}+\overline{Y}-\overline{Z})(a\overline{X}^2+b\overline{Y}^2+c\overline{X}\overline{Y}+d\overline{X}\overline{Z}+e\overline{Y}\overline{Z}+f\overline{Z}^2)=\\
&\overline{X}^2\overline{Y}+\overline{X}\overline{Y}^2-\overline{X}^2\overline{Z}+(\delta-3)\overline{X}\overline{Y}\overline{Z}+2\overline{X}\overline{Z}^2-\overline{Y}^2\overline{Z}+2\overline{Y}\overline{Z}^2-\overline{Z}^3.
\end{align*}
Then,  comparing  the coefficients on the left- and the right-hand side, we get $a=b=0$, $c+e=-d=f=1$. This leads to $\delta=1$, a contradiction again.
\end{itemize}

\noindent \textbf{Case 2:} $f_0 \neq 0$. This is equivalent to $\delta\varepsilon\neq1$. Consider  $h(X,Y):=f_0^{-1} f(X,Y)$.\\
If $\mathcal{Q}$ has a linear component, then 
$$h(X,Y)=l(X,Y)\cdot k(X,Y),$$ 
with 
\begin{equation*}
l(X,Y)=aX+bY+c \quad \textnormal{and} \quad  k(X,Y)=k_1XY^2+k_2X^2Y+k_3XY+k_4X+k_5Y+k_6
\end{equation*} 
both belonging to $\F_{q}[X,Y]$. We will show that either $a \neq 0$ and $b=0$ or $a=0$ and $b \neq 0$. Suppose first $a \neq 0$. Since $h(X,Y)$ has a term in $X^2Y^2$ and has no term in $XY^3$ and in $X^3Y$, we get that $k_1$ cannot be equal to zero and that $k_2=b=0$. A similar argument applies if $b \neq 0$, getting $a=0$.
Since $h(X,Y)=h(Y,X)$, we obtain that 
\[
h(X,Y)=(X-r)(Y-r) \cdot m(X,Y)
\]
for some $r \in \F_q$, and $m(X,Y) \in \F_q[X,Y]$ with $\deg\, m(X,Y)=2$.
Clearly,  $h(r,Y)$ is the zero polynomial.

As a consequence, all the coefficients in $Y$ of  the following polynomial are zero:
\begin{equation*} 
\begin{split}
f(r,Y)=f_0 \cdot h(r,Y)=&\, (1-r) (\delta^2  \varepsilon  - 1) (\delta^2  \varepsilon   r - \delta -  r + 1) +\\
&(\delta ^4 \varepsilon ^2+\delta ^3 \varepsilon -3 \delta ^2 \varepsilon -\delta +\delta ^4 r^2 \varepsilon ^2+\delta ^3 r^2 \varepsilon ^2-3 \delta ^2 r^2 \varepsilon -\delta  r^2 \varepsilon+\\
&2 r^2+\delta ^6 r \varepsilon ^3-5 \delta ^4 r \varepsilon ^2+6 \delta ^2 r \varepsilon +\delta  r \varepsilon +\delta  r-4 r+2)Y+\\
&(1-r) (\delta ^2 \varepsilon -1)(\delta  r \varepsilon -\delta ^2 \varepsilon -r+1)Y^2.
\end{split}
\end{equation*}

In particular, the constant term is zero and hence either $r=1$ or $r=(\delta-1)/(\delta^2\varepsilon-1)$.
If $r=1$ then the coefficient of $Y$ is $\varepsilon \delta^3 (\delta^3\varepsilon^2 - 3\delta\varepsilon + \varepsilon + 1)$ which cannot be zero under our hypothesis.
If $r=(\delta-1)/(\delta^2\varepsilon-1)$ then the coefficient of $Y$ is $\delta ^2 \left(\delta ^3 \varepsilon ^2-3 \delta  \varepsilon +\varepsilon +1\right)$, which again cannot be zero.\\
This concludes the proof.
\end{proof}

Since a cubic curve without linear components is absolutely irreducible, Lemma \ref{no-linear-components} implies the following result.
\begin{proposition}\label{cubic-irreducible}
    Let $\delta,\varepsilon \in \F_{q}^*$ such that $\delta \varepsilon = 1$ and assume $\delta^2-2\delta +1 \neq 0$. 
    Then, the curve $\mathcal{Q}: f(X,Y)=0$ (cf. \ \eqref{equazione curva}) is an absolutely irreducible cubic.
\end{proposition}

\begin{lemma}\label{no-quadratic}
Let $\delta,\varepsilon \in \F_{q}^*$ satisfy $\delta^2 \varepsilon \neq 1$ and
$\delta^3\varepsilon^2 + (1 -3\delta)\varepsilon + 1  \neq 0$. Moreover, assume that one of the following holds:
\begin{itemize}
    \item [$i)$] $\varepsilon \neq 1$,
    \item [$ii)$] $\varepsilon =1 $ and either $q$ is an odd power of $2$, or $q \equiv 0, 2,3 \pmod 5$ odd.
\end{itemize}
Then, the curve $\mathcal{Q}: f(X,Y)=0$  (cf. \eqref{equazione curva}) has no quadratic components.
\end{lemma}
\begin{proof}
If $\mathcal{Q}$ is a cubic, then the thesis follows from Proposition \ref{cubic-irreducible}. So, we assume that $\mathcal{Q}:f(X,Y)=0$ is a quartic. In particular, $f_0\neq 0$, which is equivalent to $\delta \varepsilon \neq 1$. Let us suppose that $\mathcal{Q}$ splits into  two absolutely irreducible conics. Hence, $f(X,Y)=a(X,Y) \cdot b(X,Y)$ where 
\begin{align*}
&a(X,Y)=a_0Y^2+a_1XY+a_2X^2+a_3Y+a_4X+a_5\\
&b(X,Y)=b_0X^2+b_1XY+b_2Y^2+b_3X+b_4Y+b_5, 
\end{align*}
both belonging to $\F_q[X,Y]$. Note that, since $f(X,Y)$ has no term in $X^4$,$Y^4$, $X^3$, $Y^3$ and in $X^3Y$ and $XY^3$ then 
\begin{equation}\label{zero-coeff}
    \begin{cases}
    a_0 b_2=0\\
    a_2 b_0=0\\
    a_4 b_0 + a_2 b_3=0\\
    a_3 b_2 + a_0 b_4=0\\
    a_1 b_0 + a_2 b_1=0\\
    a_0 b_1 + a_1 b_2=0
    \end{cases}
\end{equation}
\textbf{Case 1.} $a_0\neq 0$. This implies in the system above  that $b_1=b_2=b_4=0$. Since $b(X,Y)$ is polynomial of degree $2$, then $b_0 \neq 0$. So, we get $a_1=a_2=a_4=0$ and hence
\begin{equation}
a(X,Y)=a_0Y^2+a_3Y+a_5 \quad \textnormal{ and } \quad  b(X,Y)=b_0X^2+b_3X+b_5.
\end{equation}
Let $h(X,Y):=f_0^{-1} \cdot f(X,Y)$. Since $f_0=a_0b_0$ (cf. \eqref{coefficients-f}), we have 
\begin{equation}
\begin{split}
h(X,Y)=( Y^2+s_1 Y+s_2)(X^2+t_1 X+t_2)
\end{split}
\end{equation}
where $s_i=a_{2i+1}/a_0$ and  $t_i=b_{2i+1}/b_0$, $i \in \{1,2\}$.
Since  $h(X,Y)=h(Y,X)$, $s_1=t_1$ and $s_2=t_2$ hold.
It follows that 
\[
s_2^2=f_5/f_0,\quad
s_1^2=f_3/f_0,\quad
s_1=f_1/f_0,\quad
s_1s_2=f_4/f_0, \quad
s_2=f_2/f_0.
\]
This implies $f_2^2=f_0f_5$, namely (cf. \eqref{coefficients-f})
\[
\delta  \left(\delta ^2 \varepsilon -1\right)^2 \left(\delta ^3 \varepsilon ^2-3 \delta  \varepsilon +\varepsilon +1\right)=0
\]
which contradicts our hypotheses.\\

\smallskip
\noindent \textbf{Case 2.} $a_0=0$. If $a_2 \neq 0$, by System  \eqref{zero-coeff} we get $b_0=b_1=b_3=0$. Then $b_2 \neq 0$ and $a_1=a_3=0$, getting 
\begin{equation}
a(X,Y)=a_2X^2+a_4X+a_5 \quad \textnormal{ and } \quad  b(X,Y)=b_2Y^2+b_4Y+b_5.
\end{equation}
Similarly, this leads to a contradiction  as discussed in \textbf{Case 1}.\\
If $a_2=0$, since $a_1 \neq 0$ we get $b_0=b_2=0$ and hence 
\begin{equation}
a(X,Y)=a_1XY+a_3Y+a_4X+a_5 \quad \textnormal{ and } \quad  b(X,Y)=b_1XY+b_3X+b_4Y+b_5 .
\end{equation}
Let $h(X,Y):=f_0^{-1} \cdot f(X,Y)$. Since $f_0=a_1b_1$ (cf. \eqref{coefficients-f}), we have 
\begin{equation}\label{factor-h}
\begin{split}
h(X,Y)=( XY+s_1 Y+s_2X+s_3)(XY+t_1X+t_2Y+t_3),
\end{split}
\end{equation}
where $s_i=a_{i+2}/a_1$ and  $t_i=b_{i+2}/b_1$, $i \in \{1,2,3\}$.
Hence, the following conditions hold:
\begin{equation}\label{s-t}
    \begin{cases}
        s_1+t_2=h_1\\
        s_2 + t_1=h_1\\
        s_2 t_1=h_2\\
        s_1t_2=h_2\\
        s_3 + s_1 t_1 + s_2 t_2 + t_3=h_3\\
        s_3 t_1 + s_2 t_3=h_4\\
        s_3 t_2 + s_1 t_3=h_4\\
        s_3t_3=h_5
        
    \end{cases}
\end{equation}
where $h_i=f_i/f_0$, $i \in \{1,2,3,4,5\}$. By \eqref{equazione curva}, $f_2 \neq 0$ and hence $h_2 \neq 0$. By System \eqref{s-t}, $s_1,s_2,t_1$ and $t_2$ are not zero.

Multiplying by $s_1$ and by $s_2$ the first and second equations  yields respectively
\begin{equation}
    s_1^2+s_1t_2=h_1s_1 \quad \textnormal{ and } \quad s_2^2+s_2t_1=h_1s_2.
\end{equation}
Subtracting the two expressions above and using  that $s_1t_2=h_2=s_2t_1$ gives
\begin{equation}\label{s}
    s_2^2-s_1^2=h_1(s_2-s_1).
\end{equation}
Similarly, by interchanging the role of $s_i$ and $t_i$, we get
\begin{equation}\label{t}
    t_2^2-t_1^2=h_1(t_2-t_1).
\end{equation} 
\textit{Case 2.1} $s_1 \neq s_2$. Then by \eqref{s},  $s_1+s_2=h_1$. By the first two equations of the System in \eqref{s-t}, we have $s_1=t_1$ and $s_2=t_2$. Moreover, by the third-last and the second-last equations in System \eqref{s-t}, we get $s_3=t_3$ and hence System \eqref{s-t} reduces  into 
\begin{equation}\label{reduced-system}
\begin{cases}
    s_1+s_2=h_1\\
    s_1s_2=h_2\\
2s_3+s_1^2+s_2^2=h_3\\
s_3(s_1+s_2)=h_4\\
s_3^2=h_5.
\end{cases}
\end{equation}
If $h_1=0$, then $f_1=f_4=0$. This implies that $\varepsilon=1$ and either $\delta=1$ or $\delta=-2$, which in each case leads to a contradiction to the assumption $\delta^3 \varepsilon^2 + (1-3\delta)\varepsilon + 1 \neq 0$. 
Hence $h_1 \neq 0$.\\
By the second-last equation   in \eqref{reduced-system}, $s_3=h_4/h_1$. Hence $h_5=s_3^2=h^2_4/h^2_1,$
getting 
\begin{equation}\label{epsilon=1}
    f_5f_1^2= f_4^2f_0.
\end{equation}
Combining the equations of  \eqref{reduced-system}, we also have that 
$$h_3=2s_3+(s_1+s_2)^2-2s_1s_2=2h_4/h_1+h_1^2-2h_2;$$
equivalently, 
\begin{equation}\label{f-exp}
    f_0f_1f_3=f_1^3 - 2f_0f_1f_2 + 2f_0^2f_4.
\end{equation}
Substituting the coefficients of $f(X,Y)$ (cf. \eqref{equazione curva}) in 
 \eqref{epsilon=1}, we get
\begin{equation*}
\delta ^2 (\varepsilon -1) \left(\delta ^2 \varepsilon -1\right)^3 \left(\delta ^3 \varepsilon ^2-3 \delta  \varepsilon +\varepsilon +1\right)=0,
\end{equation*}
that implies $\varepsilon=1$,
and combining it with \eqref{f-exp}
$$(\delta -1)^6 \delta ^2 (\delta +1)^2 (\delta +2) \left(\delta ^2+3 \delta +1\right)=0.$$
In our assumptions, the expression above is zero if and only if $\delta^2+3\delta +1 = 0$. Then, if $q=2^{2m+1}$ for some non-negative integer $m$ or $q \equiv 2,3 \pmod 5$ odd, the equation $x^2+3x+1=0$ has no roots in $\F_q$. If $q \equiv 0 \pmod 5$ odd, by hypotheses $\delta \neq 1$ and hence $\delta^2+3\delta+1\neq 0$. Then, in any case, we get a contradiction.\\
\textit{Case 2.2} $s_1=s_2$. By System \eqref{s-t}, $t_1=t_2$. Now, since $s_1+t_1=h_1$ and $s_1t_1=h_2$, we obtain that $s_1$ and $t_1$ are the solutions of the equation:
\[
x^2-\frac{2-\delta^2\varepsilon-\delta\varepsilon}{\delta\varepsilon-1}x+\frac{\delta^2\varepsilon-1}{\delta\varepsilon-1}=0.
\]
Let $\xi:=\delta\varepsilon-1$ and $\eta:=\delta^2\varepsilon-1$. Then, the equation can be written as
\[
\xi x^2+(\xi + \eta)x+\eta=0;
\]
whose solutions are $-1$ and $-\eta/\xi=\frac{1-\delta^2\varepsilon}{\delta\varepsilon-1}$.\\
By \eqref{factor-h}, assuming $s_1=s_2=-1$ and $t_1=t_2=\frac{1-\delta^2\varepsilon}{\delta\varepsilon-1}$ is equivalent to assuming $t_1=t_2=-1$ and $s_1=s_2=\frac{1-\delta^2\varepsilon}{\delta\varepsilon-1}$. Then, without loss of generality, we may choose $s_1=s_2=-1$ and   $t_1=t_2=\frac{1-\delta^2\varepsilon}{\delta\varepsilon-1}$.\\
Let us consider the equations $s_3 t_1 + s_2 t_3=h_4$ and $s_3t_3=h_5$, we have $t_3=s_3 \frac{1-\delta^2 \varepsilon} {\delta \varepsilon -1}-h_4$, hence $s_3$ is solution of the following equation 

$$\frac{1-\delta^2 \varepsilon}{\delta \varepsilon -1} x^2-h_4x-h_5=0,$$
that is equivalent to
$$(1-\delta^2 \varepsilon)x^2+(\delta ^2 \varepsilon +\delta -2)x+(1-\delta)=0.$$
We get that either $s_3=1$ or $s_3=\frac{1-\delta}{1-\delta^2 \varepsilon}$. If $s_3=1$, the equation $s_3 + s_1 t_1 + s_2 t_2 + t_3=h_3$ reads as
\[
\delta ^3 \varepsilon  \left(\delta ^3 \varepsilon ^2-3 \delta  \varepsilon +\varepsilon +1\right)=0
\]
against our hypothesis.\\
Similarly, if $s_3=\frac{1-\delta}{1-\delta^2 \varepsilon}$,  the equation $s_3 + s_1 t_1 + s_2 t_2 + t_3=h_3$ reads as
$$
\delta ^2 \varepsilon  \left(\delta ^2 \varepsilon -1\right)=0
$$
in contradiction with our hypothesis.
This concludes the proof.
\end{proof}

\begin{remark}\label{remark: conic} Note that if $\varepsilon =1$ and either $q$ is an even power of $2$ or  $q \equiv 1,4 \pmod 5$  odd, by the proof of Lemma \ref{no-quadratic} \textit{Case 2.1}, the conic with equation 
\begin{equation}\label{conic}
XY-(\delta+1)X-Y+1=0
\end{equation}
is a component of $\mathcal{Q}$ where $\delta \in \F_q$ satisfies $\delta^2+3\delta+1=0$.
\end{remark}
We are now ready to prove Theorem \ref{t:conj}.
\begin{proof}[Proof of Theorem \ref{t:conj}]
Assume first that $\varepsilon \neq 1$ or that $\varepsilon=1$ and $q$ either an odd power of $2$ or $q \equiv 0,2,3 \pmod 5$ odd. Then, from Lemma \ref{quartic curve}, \ref{no-linear-components},  \ref{no-quadratic} and Proposition \ref{cubic-irreducible}, the variety $\mathcal{V}$ is an absolutely irreducible curve of genus at most three.
The Hasse-Weil Theorem implies that for $q\geq 37$ the variety $\mathcal{V}$ has an $\Fq$-rational point.
For $q<37$, the result has been directly checked by a GAP script, available at \href{https://pastebin.com/PHQJnAq0}{https://pastebin.com/PHQJnAq0}.\\
If $\varepsilon =1$ and either $q$ is an even power of $2$ or $q \equiv 1,4 \pmod 5$ odd, we deduce the result from Remark~\ref{remark: conic}. 
In this case, we claim indeed that $\mathcal{Q}$ has an affine point of the form $(\ell,\ell^{q^s})$ and using such point, we are able to show explicitly an affine point of type $(\ell,\ell^{q^s},\ell^{q^{2s}},\ell^{q^{3s}},\ell^{q^{4s}})$ of $\mathcal{C}$ (cf.\ \eqref{equazioni}) which corresponds to an $\mathbb{F}_q$-rational point of $\mathcal{V}$. 
Putting $Y=X^{q^s}$ in \eqref{conic}, we obtain
\[
X^{q^s+1}-(\delta+1)X-X^{q^s}+1=0.
\]
This has a solution if and only if $\xi^{q^s+1}-\delta\xi-\delta=0$, where $\xi=X-1$.
To show that this equation has a solution in $\mathbb F_q$ we apply \cite[Theorem 8]{SheekeyMcGuire}.
Define 
\[
M=\begin{pmatrix}
0&\delta\\
1&\delta
\end{pmatrix}^5.
\]
If $q=2^{2m}$ for some positive integer $m$, then the $(2,2)$-entry of $M$ (that is $G_5$  in the notation of \cite{SheekeyMcGuire}) is equal to $\delta ^3 (\delta +1) (\delta +3)$ and lies in $\F_q$. Hence, by \cite[Theorem 8]{SheekeyMcGuire}, the equation $\xi^{q^s+1}-\delta\xi-\delta=0$ has at least one root.

If $q \equiv 1,4 \pmod 5$ is odd, \cite[Theorem 8]{SheekeyMcGuire} shows that the same equation has a solution if and only if
\[
\Delta=\mathrm{Tr}(M)^2-4\det(M)
\]
is a square in $\F_q$. A direct computation gives
\[
\Delta=\delta^5(\delta+4)(\delta^2+3\delta+1)^2=0,
\]
so $\Delta$ is a square in $\F_q$, and again the equation has a solution. Therefore there exists a point of the form $(\ell,\ell^{q^s})$ on the conic \eqref{conic}, and hence on $\mathcal{Q}$.

By \eqref{conic}, $\ell \neq 1$ and
\begin{equation}\label{lqs}
\ell^{q^s}=\frac{(\delta+1)\ell-1}{\ell-1}.
\end{equation}
By the expression above,
\begin{equation}\label{lq2s}
    \ell^{q^{2s}}=\frac{(\delta+1)\ell^{q^s}-1}{\ell^{q^s}-1}=\frac{(\delta +2) \ell-1}{\ell}.
\end{equation}
By Lemma \ref{quartic curve}, $\mathcal{Q}$ is birationally equivalent to $\cC$ and a point $(\ell,\ell^{q^s})$ of $\mathcal{Q}$ corresponds to a point $(\ell,\ell^{q^s},\bar{c},\bar{d},\bar{e})$ of $\cC$ with 
\begin{equation}\label{primabarc}
\bar{c}=\frac{(\ell^{q^s} \delta^2    - \ell^{q^s} - \delta + 1)}{\ell( \ell^{q^s} \delta    -  \ell^{q^s} -  \delta^2    + 1 )}, \quad  \bar{d}= \frac{(\ell^{q^s+1} \bar{c} \delta    - \ell^{q^s+1} + \ell- \ell^{q^s} \delta^2    + \ell^{q^s} - 1)}{\ell^{q^s+1}\bar{c} \delta   },
\end{equation}
and $\bar{e}=-\frac{1}{\ell^{q^s+1}\bar{c}\bar{d}}$ (cf.\ \eqref{variableC}, \eqref{variableD}, and \eqref{variableE}). Substituting Equation \eqref{lqs} in the expression of $\bar{c}$, we have that
\begin{equation}\label{barc}
    \bar{c}=\frac{(\ell^{q^s} \delta^2    - \ell^{q^s} - \delta + 1)}{\ell( \ell^{q^s} \delta    -  \ell^{q^s} -  \delta^2    + 1 )}=\frac{(\delta +2) \ell-1}{\ell},
\end{equation}
getting $\bar{c}=\ell^{q^{2s}}$. By \eqref{lqs} and \eqref{lq2s}, $$\ell^{q^{3s}}=\frac{-\delta +\delta ^2 \ell +3 \delta  \ell +\ell -1}{\delta  \ell +\ell -1},$$
while by \eqref{primabarc}, \eqref{lqs} and \eqref{barc},
$$\bar{d}=\frac{\delta +\delta ^2 \ell ^2+3 \delta  \ell ^2+\ell ^2-\delta ^2 \ell -3 \delta  \ell -2 \ell +1}{(\delta  \ell +\ell -1) (\delta  \ell +2 \ell -1)}.$$
Let us consider
\begin{equation}
    \bar{d}-\ell^{q^{3s}}=-\frac{\left(\delta ^2+3 \delta +1\right) \ell }{\delta  \ell +2 \ell -1}.
\end{equation}
Since $\delta^2+3 \delta +1 =0$, it follows that $\bar{d}=\ell^{q^{3s}}.$ 
An analogous computation shows that $\bar{e}=\ell^{q^{4s}}$, which completes the proof. 
\end{proof}

\section{\texorpdfstring{On the case $\rk A=\rk B=4$}{On the case rk A=rk B=4}}\label{s:rkAB4}

\subsection{Geometric description}\label{ss:geo}
Now assume that $\wp_{\Gamma}(\Sigma)$ is a MSLS not of pseudoregulus type, that $\Sigma$ is the set of all points $[x_0,\ldots,x_4]$ of $\PG(4,q^5)$ with coordinates rational 
over $\Fq$, and that $\Gamma$ has $\rk A=\rk B=4$
for any choice of $\sigma$. Thus $\sigma:(x_0,\ldots,x_4)\mapsto(x_0^q,\ldots,x_4^q)$ may be assumed.
Let $S_A=\la A,A_1,\ldots,A_4\ra$ and $S_B=\la B,B_1,\ldots,B_4\ra$.

It may be assumed that $S_A$ and $S_B$ have equations $x_0=0$ and $x_4=0$, respectively.
Since the points $A$ and $B$ have rank four,
\begin{equation}\label{e:AB4}
A=[0,a_1,a_2,a_3,1]\text{ and }B=[1,b_1,b_2,b_3,0].
\end{equation}
The dependence of $A$, $A_1$, $B$, and $B_2$ gives
\begin{equation}\label{e:cpln}
    \rk
    \begin{pmatrix}
    a_1^q-a_1&a_2^q-a_2&a_3^q-a_3\\
    b_1^{q^2}-b_1&b_2^{q^2}-b_2&b_3^{q^2}-b_3
    \end{pmatrix}=1,
\end{equation}
and $G$ of coordinates $[0, a_1^q-a_1,a_2^q-a_2,a_3^q-a_3,0]$ is the intersection of the lines $\la A,A_1\ra$, and $\la B,B_2\ra$.

The coordinates of the solids joining $\Gamma$ and the points of $\Sigma$ are the minors of
\[
\begin{pmatrix}
0&a_1&a_2&a_3&1\\ 1&b_1&b_2&b_3&0\\ 0&a_1^q-a_1&a_2^q-a_2&a_3^q-a_3&0\\ u_0&u_1&u_2&u_3&u_4
\end{pmatrix},\quad (u_0,\ldots,u_4)\in(\mathbb F_q^5)^*.
\]
Intersecting with the line $x_0=x_1=x_4=0$ gives the following form for the linear set:

\begin{align*}
\left\{\la \left( \right. \right.&  m_3(x_1+b_1x_0+a_1x_4)-m_1(x_3+b_3x_0+a_3x_4),
m_2(x_1+b_1x_0+a_1x_4)\\
&\left. \left.-m_1(x_2+b_2x_0+a_2x_4) \right)\ra_{\Fqc}\colon (x_0,\ldots,x_4)\in\F_q^5\setminus\{(0,\ldots,0)\}
\right\},
\end{align*}
where $(m_1,m_2,m_3)=( a_1^q-a_1,a_2^q-a_2,a_3^q-a_3)$.
This can be seen as
\[
\begin{pmatrix}
m_3&0&-m_1\\ m_2&-m_1&0
\end{pmatrix}
\left[
x_0
\begin{pmatrix}
   b_1\\ b_2\\ b_3 
\end{pmatrix}
+x_4
\begin{pmatrix}
   a_1\\ a_2\\ a_3 
\end{pmatrix}
+\begin{pmatrix}
    x_1\\ x_2\\ x_3
\end{pmatrix}\right],
\]
that is, 
as the projection from the vertex $M=[m_1,m_2,m_3]$ of the linear set in $\PG(2,q^5)$ defined by the $\Fq$-subspace
\begin{equation}\label{e:Uper4}
    U_{\mathbf a,\mathbf b}=\la\mathbf a,\mathbf b\ra_{\Fq}+\F_{q}^3,
\end{equation}
where $\mathbf a=(a_1,a_2,a_3)$, $\mathbf b=(b_1,b_2,b_3)$.

Summarizing:
\begin{theorem}\label{paperino}
  Let $\mathbb L=\wp_\Gamma(\Sigma)$ be a maximum scattered linear set in $\PG(1,q^5)$ 
  not of pseudoregulus type, and assume that $\rk A=\rk B=4$.
  Then $\mathbb L$ can be represented as a projection from the vertex $M=[m_1,m_2,m_3]$ of the linear set in $\PG(2,q^5)$ defined by the $\Fq$-subspace $U_{\mathbf a,\mathbf b}$ in \eqref{e:Uper4}.
  Conversely, if $M=[ a_1^q-a_1,a_2^q-a_2,a_3^q-a_3]$ has rank three in $\PG(2,q^5)$ and \eqref{e:cpln} holds, then the projection from $M$ of $U_{\mathbf a,\mathbf b}$ is, if maximum scattered, a linear set in $\PG(1,q^5)$ 
  not of pseudoregulus type such that $\rk A=\rk B=4$.
\end{theorem}

Note that if $M$ has rank three, then $a_1^q-a_1$, $a_2^q-a_2$, $a_3^q-a_3$ are $\Fq$-linearly dependent, and this implies that $A=[0,a_1,a_2,a_3,1]$ has rank four.

\subsection{Algebraic description}\label{ss:alg}

\begin{sloppypar}
As above, $s=1$ may be assumed.
For any two distinct points 
$P_i=\la(x^{(i)}_0,x^{(i)}_1,x^{(i)}_2,x^{(i)}_3,x^{(i)}_4)_\cB\ra_{\Fqc}$,
$i=1,2$, in $\Sigma$, the points $A$, $A_1$, $B_2$, $P_1$, $P_2$ are independent;
equivalently,
\[
\begin{vmatrix} 1&-\lambda^{q^2}&\mu^{q^2}\\ x_2^{(1)}&x_3^{(1)}&x_4^{(1)}\\
x_2^{(2)}&x_3^{(2)}&x_4^{(2)}\end{vmatrix}\neq0.
\]
\end{sloppypar}
This implies that the following linear set is scattered, and projectively equivalent to $\mathbb L=\wp_{\Gamma}(\Sigma)$:
\begin{equation}\label{fgenerale}
  L=\{\la(\mu^{q^2} x_2-x_4,\mu^{q^2}x_3+\lambda^{q^2}x_4)\ra_{\Fqc}\colon
  \la(x_0,\ldots,x_4)_\cB\ra_{\Fqc}\in\Sigma\}.
\end{equation}

By \eqref{howe} and Proposition~\ref{prop:linearcombination}, it may be assumed that $\mu=1$.

From \eqref{eq23a} and $e=0$:
\[
a=b^{q^4+q^2+1},\ c=b^{-q^3},\ d=-b^{q^2+1},\ \N_{q^5/q}(b)=-1.
\]
A $w\in\Fqc^*$ exists such that $b=-w^{1-q^3}$.
Hence
\begin{equation}\label{e:abcdw}
a=-w^{q^4-q^3},\ b=-w^{1-q^3},\ c=-w^{q-q^3},\ d=-w^{q^2-q^3}.
\end{equation}
If $\lambda\neq1$, then, by \eqref{coord-u4}, $w=\lambda-1$ may be assumed.

The representing vectors for $\Sigma$ are obtained from $M_1X^q=X$, where
\begin{equation}\label{e:m1}
M_1=\begin{pmatrix}
0&0&0&-w^{q^4-q^3}&1\\ 1&0&0&-w^{1-q^3}&-\lambda\\ 0&1&0&-w^{q-q^3}&1\\ 
0&0&1&-w^{q^2-q^3}&-\lambda^{q^2}\\ 0&0&0&0&1
\end{pmatrix}
\end{equation}
By solving the equations in term of $x_3=t$ and $x_4=\theta$,
one obtains $\theta\in\Fq$, and
\begin{align}
x_0&=-w^{q^4-q^3}t^q+\theta\notag\\
x_1&=-w^{1-q^4}t^{q^2}-w^{1-q^3}t^q+(1-\lambda)\theta\notag\\
x_2&=-w^{q-1}t^{q^3}-w^{q-q^4}t^{q^2}-w^{q-q^3}t^q+(2-\lambda^q)\theta\label{e:x2}\\
&-w^{q^2}(w^{-q}t^{q^4}+w^{-1}t^{q^3}+w^{-q^4}t^{q^2}+w^{-q^3}t^q+w^{-q^2}t)
+2(1-\lambda^{q^2})\theta=0.\notag
\end{align}
The last equation is equivalent to 
\begin{equation}\label{e:x3}
w^{q^2}\tr_{q^5/q}(w^{-q^2}t)=2(1-\lambda^{q^2})\theta.
\end{equation}

The condition $\rk B=4$ implies that the rank of the following matrix containing the coordinates of $v$, $v_1$, $v_2$, $v_3$, $v_4$ with respect the basis $\mathcal B$ is four:
\[
C=\begin{pmatrix}
1&-\lambda&1&-\lambda^{q^2}&1\\ 
-\lambda^{q^3}a+1&-\lambda^{q^3}b-\lambda+1&-\lambda^{q^3}c+1-\lambda^q&1-\lambda^{q^3}d-\lambda^{q^2}&1\\
0&0&1&-\lambda^{q^2}&1\\
-\lambda^{q^3}a+1&-\lambda^{q^3}b-\lambda&-\lambda^{q^3}c+1&1-\lambda^{q^3}d-\lambda^{q^2}&1\\
0&0&0&0&1
\end{pmatrix}.
\]
The rank of $C$ is four if and only if
\begin{equation}\label{e:rkabcd}
\lambda^{q^3+q^2+q+1}a+\lambda^{q^3+q^2+q}b+\lambda^{q^3+q^2}c+\lambda^{q^3}d-1=0.
\end{equation}
For $\lambda=1$ this leads to $a+b+c+d-1=0$, equivalent by \eqref{e:abcdw} to $\tr_{q^5/q}(w)=0$.
For $\lambda\neq1$, since $w=\lambda-1$, the equations \eqref{e:abcdw} are equivalent to
\[
a=-\frac{\lambda^{q^4}-1}{\lambda^{q^3}-1},\ 
b=-\frac{\lambda-1}{\lambda^{q^3}-1},\ 
c=-\frac{\lambda^{q}-1}{\lambda^{q^3}-1},\ 
d=-\frac{\lambda^{q^2}-1}{\lambda^{q^3}-1}.
\] 
Combining with \eqref{e:rkabcd}, $\N_{q^5/q}(\lambda)=1$ results.

The following straightforward result will be used in the proof of the Theorem~\ref{t:fecc}.
\begin{proposition}\label{p:sostituz}
Let $\rho\in\Fqc$ such that $\tr_{q^5/q}(\rho)\neq0$; 
let $g:\Fqc\rightarrow\Fqc$ be an $\Fq$-linear map such that $\ker g=\Fq$ and $\tr_{q^5/q}\circ g$ is the zero map.
Also let
\begin{align*}
S&=\{(\theta,y)\colon \theta\in\Fq,\,y\in\Fqc,\,\tr_{q^5/q}(y)=0\},\\
T&=\{(\theta,y)\colon \theta\in\Fq,\,y\in\Fqc,\,\tr_{q^5/q}(y)+2\theta=0\}.
\end{align*}
Then the maps
\begin{align}
&\alpha:\Fqc\rightarrow S,\ x\mapsto(\tr_{q^5/q}(\rho x),g(x)),\label{e:alpha}\\ 
&\beta:\Fqc\rightarrow T,\ x\mapsto(\tr_{q^5/q}(x), -x-x^{q^2})\label{e:beta}
\end{align}
are well-defined and bijective.
\end{proposition}

\begin{theorem}\label{t:fecc}
If $\rk A=\rk B=4$, then $\mathbb L$ is equivalent to $\mathbb E$ or $L_{F}$, where
\begin{align}\mathbb E&=
\{\la(\eta(x^q-x)+\tr_{q^5/q}(\rho x), x^q-x^{q^4})\ra_{\Fqc}\colon x\in\F_{q^5}^*\},\label{e:fecc1}\\ 
&\eta\neq0,\ \tr_{q^5/q}(\eta)=0\neq\tr_{q^5/q}(\rho),\notag\\
F(x)&=k(x^{q}+x^{q^3})+x^{q^2}+x^{q^4},\quad \N_{q^5/q}(k)=1.\label{e:fecc2}
\end{align}
\end{theorem}
\begin{proof}
Let $\eta=w^{q^2}$ and $y=tw^{-q^2}$.
Note that \eqref{e:x3} is equivalent to $\tr_{q^5/q}(y)=0$ for $\lambda=1$,
and to
\begin{equation}\label{e:trteta}
\tr_{q^5/q}(y)+2\theta=0
\end{equation}
for $\lambda\neq1$.
By combining \eqref{fgenerale}, \eqref{e:x2}, the elements of $\mathbb L$ are of type
\[
\la(\eta y+\lambda^{q^2}\theta,-\eta^{q^4}(y^q+y^{q^2}+y^{q^3})+(1-\lambda^q)\theta)\ra_{\Fqc}.
\]
For $\lambda=1$, it follows
\begin{equation}
\{\la(\eta y+\theta,y+y^{q^4})\ra_{\Fqc}\colon \theta\in\Fq,\,y\in\Fqc,\,\tr_{q^5/q}(y)=0\}.
\label{e:lambdae1}
\end{equation}
The form \eqref{e:fecc1} can be obtained by the substitution $(\theta,y)=(\tr_{q^5/q}(\rho x),x^q-x)$ in \eqref{e:lambdae1} according to Proposition~\ref{p:sostituz}.

For $\lambda\neq1$, by \eqref{e:trteta}, the pairs are of type
\[
(\eta y+(\eta+1)\theta,\eta^{q^4}(y+y^{q^4}+\theta)).
\]

Now let us transform the pairs as follows:
\[  
\begin{pmatrix}0&\eta^{-q^4}\\ 1&-\eta^{-q^4+1}-\eta^{-q^4}\end{pmatrix}
\begin{pmatrix}
    \eta y+(\eta+1)\theta\\ \eta^{q^4}(y+y^{q^4}+\theta)
\end{pmatrix}
=
\begin{pmatrix}
    y+y^{q^4}+\theta\\ -y-ky^{q^4}
\end{pmatrix}
\]
for $k=1+\eta=\lambda^{q^2}$.
By substituting $(\theta,y)=\beta(x)$ we get
$(x^{q^3},x+x^{q^2}+k(x^{q^4}+x^q)$, equivalent to $L_F$.

The norm of $k$ is one as it has been noted as a consequence of \eqref{e:rkabcd}.
\end{proof}

\begin{proposition}\label{p:fecc}
If $\tr_{q^5/q}(\eta^{-1})\neq0$, then the linear set $\mathbb E$ in \eqref{e:fecc1} is projectively equivalent to $L_{F_1}$ where
\begin{equation}\label{e:fecc3}
F_1(x)=\tr_{q^5/q}(\eta^{-1})x^{q^4}-(\eta^{-q^4}+\eta^{-1})\tr_{q^5/q}(x).
\end{equation}
\end{proposition}
\begin{proof}
Let $\rho=\eta^{-1}$, and
\[
g(x)=\rho\left(\tr_{q^5/q}(\rho) x-\tr_{q^5/q}(\rho x)\right).
\]
Clearly $\Fq\subseteq\ker g(x)$, and $\tr_{q^5/q}(g(x))=0$ for all $x\in\Fqc$.
Taking into account \eqref{e:lambdae1},
for $\theta=\tr_{q^5/q}(\rho x)$ and $y=g(x)$ we have
\[
\eta y+\theta=\tr_{q^5/q}(\rho) x.
\]
In particular, the sum of $\eta g(x)$ with an $\Fq$-linear map of rank one is bijective, and this implies that the rank of $g(x)$ is at least four.
Therefore, $\ker g=\Fq$ and the hypotheses of Proposition~\ref{p:sostituz} are satisfied.
The substitution $y=g(x)$ in $y+y^4$, neglecting the first-degree term, gives 
\[
\left(\tr_{q^5/q}(\rho) x,\eta^{-q^4}\tr_{q^5/q}(\eta^{-1})x^{q^4}-(\eta^{-q^4}+\eta^{-1})\tr_{q^5/q}(\eta^{-1}x)\right)
\]
and leads to \eqref{e:fecc3}.
\end{proof}

\begin{remark}\label{r:gap2}
The linear sets of type \eqref{e:fecc1} and \eqref{e:fecc2} have been analyzed by a GAP script for $q\le25$.
None of them is scattered.
The script that performed this check can be found at the URL \href{https://pastebin.com/TjrKuu0z}{https://pastebin.com/TjrKuu0z}.\\
In particular, for $q \leq 9$ and $\mathrm{N}_{q^5/q}(k)= 1 \neq k$, the linear set $L_F$ has points with weight at most 2, i.e., each ratio $F(x)/x$ occurs for at most $q^2-1$ distinct non-zero values of $x \in \F_{q^5}^*.$ On the other hand, for $k=1$, $L_F$ is equivalent to the linear set associated with the trace function $\mathrm{Tr}_{q^5/q}(x)$, see \cite{CMP}.
\end{remark}

\section{Conclusion}

In this paper, starting from the results in \cite{MZ2019}, we proved that the maximum scattered linear sets of $\PG(1,q^5)$ are always of LP type if the rank of $A$ or $B$ is five. The case in which $\rk A =\rk B =4$ remains open. In this case, we can describe their form and/or the polynomial that defines them. Any linear set with this shape would be a new type. We have not established their existence and conjecture that they do not exist; in any case, we have ruled out their existence for $q\le 25$. The following theorem summarizes our work.
To provide the reader with more insight, we extend the proof to include a description of the steps that led to the result.
\begin{theorem}\label{t:summary}
If $\mathbb L$ is a maximum scattered linear set in $\PG(1,q^5)$ not of pseudoregulus type, then $\mathbb L$ is the projection of a canonical $\Fq$-subgeometry $\Sigma\subseteq\PG(4,q^5)$ from a plane $\Gamma$ such that $\Gamma\cap\Sigma=\emptyset$.
Let  $\sigma$ be a generator of the subgroup $\mathbb G$ of $\PGaL(5,q^5)$ fixing $\Sigma$ pointwise.
Define $A=\Gamma\cap\Gamma^{\sigma^4}$ and $B=\Gamma\cap\Gamma^{\sigma^3}$. Then, $A$ and $B$ are points.
If $\rk A=5$ or $\rk B=5$, then $\mathbb L$ is of LP type.
Otherwise $q>25$, $\mathbb L$ is not of LP type and is, up to equivalence, of shape \eqref{e:fecc1} or $L_F$ where $F(x)$ is as in \eqref{e:fecc2}.
\end{theorem}

\noindent \emph{Proof and summary of the paper.}\quad Let $\mathbb L$ be a maximum scattered linear set of $\PG(1,q^5)$, not of pseudoregulus type. 
By Theorems~\ref{t:lupo} and \ref{t:CsZa20162}, $\mathbb L$ is the projection of a canonical $\Fq$-subgeometry $\Sigma$ of $\PG(4,q^5)$ from a plane $\Gamma$. The group $\mathbb G$ of collineations fixing $\Sigma$ pointwise is a cyclic group of order five. 
By Theorem~\ref{t:CsZa20162}, the intersection $\Gamma\cap\Gamma^\sigma$ is a point of $\PG(4,q^5)$ for every generator $\sigma$ of $\mathbb G$. 
This allows us to define two points $A=\Gamma\cap\Gamma^{\sigma^4}$ and $B=\Gamma\cap\Gamma^{\sigma^3}$ and the elements of their orbits $A_i=A^{\sigma^i}$, $B_i=B^{\sigma^i}$, $i=1,2,3,4$.

The point $A=A_{\sigma,\Gamma}$ depends on the vertex $\Gamma$ and the generator $\sigma$ of the collineations group $\mathbb G$, and the proven properties hold for every $\sigma$. 
At the beginning of Section~\ref{s:class}, we show that if $\rk B=5$, then by replacing $\sigma$ with another generator $\tau$ if necessary, we have $\rk A_{\tau,\Gamma}=5$. For this reason, if $\max\{\rk A,\rk B\}=5$, it may be assumed without loss of generality that $\rk A=5$.

Various properties of these points have been studied in Section~\ref{s:geomprop}. 
Overall, the ten points of the orbits of $A$ and $B$ under the action of $\mathbb G$ are distinct, and precisely four of them belong to every 
$\Gamma^{\sigma^i}$. 
The linear set $\mathbb L$ is of type LP if, and only if, the line $A_2A_4$ intersects $\Gamma$ (configuration I) or the line $B_3B_4$ intersects $\Gamma$ (configuration II), cf.\ Theorem~\ref{LPlines}. 
In Section~\ref{s:equations}, we studied the properties of the vectors representing the above points, linked by equations \eqref{linearcombination} and \eqref{viui}. 
We then found properties of $\mathbb L$, expressed in terms of vector coordinates of a basis associated with points $A$, $A_1$, $A_2$, $A_3$, $B_4$. 
Using these coordinates, equations \eqref{e:23c} allow us to find coordinates of point $A_4$ in the case $\N_{q^5/q}(\mu) \neq 1$, which leads to a geometric property of fundamental importance in this work: The line $A_3B_1$ intersects $\Gamma$. 
As a consequence of this,  in Proposition~\ref{p:almost} we have shown that if $\mathbb L$ is not of type LP and $\rk A=5$, then, up to projectivities,
\[
 \mathbb L=\{\la(x-k^{-1}x^{q^{2s}},x^{q^s}-\delta k^{q^{4s}+q^{2s}}x^{q^{2s}})\ra_{\Fqc}\colon
  x\in\Fqc^*\}
\]
for some $k\in\Fqc^*$ and $\delta\in\Fq$. 
In Theorem~\ref{t:almost}, we saw that, for $\varepsilon=\N_{q^5/q}(k)$, it follows that 
\begin{equation}\label{e:2diseq}
\begin{cases}
\delta^2\varepsilon\neq1\\
\delta^3\varepsilon^2+(1-3\delta)\varepsilon+1\neq0,
\end{cases}
\end{equation}
and the system of equations
 \[
\begin{cases}
  \varepsilon\N_{q^5/q}(x)=-1\\
  \delta^2\varepsilon x^{q^s}-\delta\varepsilon x^{q^{2s}+q^s+1}(1-x)^{q^{3s}}+(1-x)^{q^s+1}
  =0
\end{cases}\qquad\qquad\text{\eqref{e:sctness}}
\]
have no solution $x\in\Fqc$. 
However, Theorem~\ref{t:conj} shows that conditions \eqref{e:2diseq} imply that the system of equations \eqref{e:sctness} admit a solution. 
The result is achieved by translating the existence of solutions into the existence of $\Fq$-rational points on an algebraic curve.
Therefore, if $\rk A=5$, then $\mathbb L$ is of type LP. 
For the reasons given above, $\mathbb L$ is of type LP also if $\rk B=5$.

By virtue of Theorem~\ref{t:fecc}, if $\rk A=\rk B=4$, then $\mathbb L$ is equivalent to the linear set described in \eqref{e:fecc1}, or to $L_F$ where $F(x)$ is the polynomial \eqref{e:fecc2}. 
The linear sets of those two types were analyzed using a GAP script for $q\le25$ and, as mentioned in Remark~\ref{r:gap2},  none of them are scattered.\hfill$\square$

\medskip

As a consequence, we have the following result.
\begin{corollary}
Any maximum scattered linear set (MSLS) in $\PG(1,q^5)$ is,  up to equivalence in $\PGaL(2,q^5)$, one of the following:
\begin{enumerate}[(C1)]
\item a MSLS of pseudoregulus type, $\{\la(x,x^q)\ra_{\Fqc}\colon x\in \F_{q^5}^*\}$,
\item a MSLS of LP type, $\{\la(x,x^{q^s}+\delta x^{q^{5-s}})\ra_{\Fqc}\colon x\in \F_{q^5}^*\}$, $s\in\{1,2\}$, $\N_{q^5/q}(\delta)\neq0,1$,
\item 
\begin{align*}&
\{\la(\eta(x^q-x)+\tr_{q^5/q}(\rho x), x^q-x^{q^4})\ra_{\Fqc}\colon x\in\F_{q^5}^*\},\notag\\ 
&\quad\eta\neq0,\ \tr_{q^5/q}(\eta)=0\neq\tr_{q^5/q}(\rho),\label{e:c3}
\end{align*}
\item 
\[
\{\la(x,k(x^{q}+x^{q^3})+x^{q^2}+x^{q^4})\ra_{\Fqc}\colon x\in\F_{q^5}^*\},\quad \N_{q^5/q}(k)=1.
\]
\end{enumerate}
\end{corollary}
Note that the classes of sets of types (C3) and (C4) might be empty, as they actually are for $q\le25$.

\section*{Acknowledgement}
The first author acknowledges the support by the Irish Research Council, grant n. GOIPD/2022/307.
The second author is supported by the Italian National Group for Algebraic
and Geometric Structures and their Applications (GNSAGA - INdAM) and by the European Union under the Italian National Recovery and Resilience Plan (NRRP) of NextGenerationEU, with particular reference to the partnership on "Telecommunications of the Future" (PE00000001 - program "RESTART", CUP: D93C22000910001).

\vspace{1cm}
\noindent
Stefano Lia,\\
Department of Mathematics and Mathematical Statistics,\\
Umeå University,\\
\texttt{stefano.lia@umu.se}

\vspace{0.5cm}

Giovanni Longobardi\\
Dipartimento di Matematica ed Applicazioni 'R. Caccioppoli',\\
Università degli Studi di Napoli Federico II,\\
Via Cintia, Monte S.Angelo I-80126 Napoli, Italy\\
\texttt{giovanni.longobardi@unina.it}

\vspace{0.5cm}

\noindent
Corrado Zanella,\\
Dipartimento di Tecnica e Gestione dei Sistemi Industriali,\\
Universit\`a di Padova\\
\texttt{corrado.zanella@unipd.it}

\end{document}